\documentclass[12pt]{amsart}
\usepackage{graphicx}
\usepackage{amsmath,amsthm,amssymb,enumerate}
\usepackage{euscript,mathrsfs}
\usepackage[citecolor=blue,colorlinks=true]{hyperref}
\usepackage[left=2cm,right=2cm,top=3.5cm,bottom=3.5cm]{geometry}
\usepackage{color}

\usepackage{soul}

\catcode`\@=11 \@addtoreset{equation}{section}

\catcode`\@=12

\allowdisplaybreaks

\newtheorem{Theorem}{Theorem}[section]
\newtheorem{Proposition}[Theorem]{Proposition}
\newtheorem{Lemma}[Theorem]{Lemma}
\newtheorem{Corollary}[Theorem]{Corollary}

\theoremstyle{definition}
\newtheorem{Definition}[Theorem]{Definition}

\newtheorem{Remark}[Theorem]{Remark}

\newcommand{\bTheorem}[1]{
\begin{Theorem} \label{T#1} }
\newcommand{\eT}{\end{Theorem}}

\newcommand{\bProposition}[1]{
\begin{Proposition} \label{P#1}}
\newcommand{\eP}{\end{Proposition}}

\newcommand{\bLemma}[1]{
\begin{Lemma} \label{L#1} }
\newcommand{\eL}{\end{Lemma}}

\newcommand{\bCorollary}[1]{
\begin{Corollary} \label{C#1} }
\newcommand{\eC}{\end{Corollary}}

\newcommand{\bRemark}[1]{
\begin{Remark} \label{R#1} }
\newcommand{\eR}{\end{Remark}}

\newcommand{\bDefinition}[1]{
\begin{Definition} \label{D#1} }
\newcommand{\eD}{\end{Definition}}

\newcommand{\Del}{\Delta_x}

\newcommand{\prst}{\mathcal{P}}

\newcommand{\vrn}{\vr_n}
\newcommand{\vun}{\vu_n} 
\newcommand{\vmn}{\vm_n}

\newcommand{\tvm}{\widetilde{\vc{m}}}
\newcommand{\bfphi}{\boldsymbol{\varphi}}

\newcommand{\bFormula}[1]{
\begin{equation} \label{#1}}
\newcommand{\eF}{\end{equation}}

\newcommand{\Ov}[1]{\overline{#1}}

\newcommand{\DC}{C^\infty_c}

\newcommand{\aleq}{\lesssim}

\newcommand{\vr}{\varrho}

\newcommand{\tvr}{\tilde \vr}
\newcommand{\tvu}{{\tilde \vu}}

\newcommand{\vu}{\vc{u}}
\newcommand{\vm}{\vc{m}}

\newcommand{\vc}[1]{{\bf #1}}

\newcommand{\Div}{{\rm div}_x}
\newcommand{\Grad}{\nabla_x}

\newcommand{\dx}{\,{\rm d} {x}}

\newcommand{\dt}{\,{\rm d} t }

\newcommand{\intQp}[1]{\int_{R^2} #1 \ \dx}

\newcommand{\intQ}[1]{\int_{{Q}} #1 \ \dx}

\newcommand{\vv}{\vc{v}}

\newcommand{\D}{{\rm d}}

\newcommand{\ep}{\varepsilon}

\newcommand{\expe}[1]{ \mathbb{E} \left[ #1 \right] }

\newcommand{\tb}{\textcolor{black}}

\def\softd{{\leavevmode\setbox1=\hbox{d}%
          \hbox to 1.05\wd1{d\kern-0.4ex{\char039}\hss}}}
\definecolor{Cgrey}{rgb}{0.85,0.85,0.85}
\definecolor{Cblue}{rgb}{0.50,0.85,0.85}
\definecolor{Cred}{rgb}{1,0,0}
\definecolor{fancy}{rgb}{0.10,0.85,0.10}

\newcommand\Cbox[2]{%
    \newbox\contentbox%
    \newbox\bkgdbox%
    \setbox\contentbox\hbox to \hsize{%
        \vtop{
            \kern\columnsep
            \hbox to \hsize{%
                \kern\columnsep%
                \advance\hsize by -2\columnsep%
                \setlength{\textwidth}{\hsize}%
                \vbox{
                    \parskip=\baselineskip
                    \parindent=0bp
                    #2
                }%
                \kern\columnsep%
            }%
            \kern\columnsep%
        }%
    }%
    \setbox\bkgdbox\vbox{
        \color{#1}
        \hrule width  \wd\contentbox %
               height \ht\contentbox %
               depth  \dp\contentbox
        \color{black}
    }%
    \wd\bkgdbox=0bp%
    \vbox{\hbox to \hsize{\box\bkgdbox\box\contentbox}}%
    \vskip\baselineskip%
}

\date{}

\begin{document}


\title[Randomness in compressible fluid flows past an obstacle]{\tb{Randomness} in  compressible fluid flows past an obstacle}

\author{Eduard Feireisl}
\address[E. Feireisl]{Institute of Mathematics AS CR, \v{Z}itn\'a 25, 115 67 Praha 1, Czech Republic
\and Institute of Mathematics, TU Berlin, Strasse des 17.Juni, Berlin, Germany }
\email{feireisl@math.cas.cz}
\thanks{The research of E.F. leading to these results has received funding from
the Czech Sciences Foundation (GA\v CR), Grant Agreement
21--02411S. The Institute of Mathematics of the Academy of Sciences of
the Czech Republic is supported by RVO:67985840. The stay of E.F. at TU Berlin is supported by Einstein Foundation, Berlin.}

\author{Martina Hofmanov\'a}
\address[M. Hofmanov\'a]{Fakult\"at f\"ur Mathematik, Universit\"at Bielefeld, D-33501 Bielefeld, Germany}
\email{hofmanova@math.uni-bielefeld.de}
\thanks{M.H. gratefully acknowledges the financial support by the German Science Foundation DFG via the Collaborative Research Center SFB1283.}


\begin{abstract}

We consider a statistical limit of solutions to the compressible Navier--Stokes system in the high Reynolds number regime in a domain exterior to a rigid body. We investigate to what extent \tb{this highly turbulent regime can be modeled} by an external stochastic perturbation, as \tb{suggested} in the related physics literature.
To this end, we interpret the statistical limit as a stochastic process on the associated trajectory space. We suppose that the 
limit process is statistically equivalent to a solution of the stochastic compressible Euler system. Then, necessarily, 
\begin{itemize}  
\item the stochastic forcing is not active -- the limit is a \tb{statistical} solution of the deterministic Euler system;
\item the solutions S--converge to the limit; 
\item if, in addition, the expected value of the limit process solves the Euler system, then the limit is deterministic and the convergence is strong in the $L^p$--sense.  
\end{itemize}
\tb{These results strongly indicate} that a stochastic forcing may not be a suitable model for turbulent randomness in 
compressible fluid flows.

\end{abstract}

\date{\today}

\maketitle

\section{Introduction}
\label{i}

A largely used approach in the mathematical studies of hydrodynamic turbulence is to add certain stochastic perturbation to the model. For instance, Yakhot and Orszak \cite{YakOrs} suggested that a stochastically perturbed Navier--Stokes system shall possess the same \emph{statistical properties} as the deterministic Navier--Stokes system with general inhomogeneous boundary conditions. It is expected that turbulence is created at the boundary and therefore by boundary conditions, and, accordingly, the study of the associated boundary layer shall play an essential role. Due to the substantial theoretical difficulties, however, it turns out to be more easily amenable to rigorous analysis to replace boundary conditions by an external stochastic forcing term. This is one of the reasons why the field of stochastic PDEs related to motion of fluids gained a massive importance in the contemporary mathematics research.

Another approach how to observe turbulent behavior of fluids is increasing the Reynolds number, meaning following the so-called vanishing viscosity regime. On the formal level, the Navier--Stokes system then approaches the Euler system and the statistical behavior \tb{of the fluid flow} shall be therefore reflected in the Euler equations in a certain sense. In the present paper, we investigate the question whether the turbulent randomness observed along the vanishing viscosity limit can be obtained from the Euler system directly by including a suitable stochastic perturbation.

As pointed out, the Euler system is supposed to capture the behavior of real (viscous) fluids in the vanishing viscosity 
or, more precisely, large Reynolds number regime. 
The viscosity solutions resulting from this process are in turn considered to be physically relevant.
Note that solutions of the (compressible) Euler system are known to develop singularities in a finite time, whereas 
the initial--value problem is ill--posed in the class of weak solutions. Relevant examples in the 
incompressible case are discussed by Bressan and Murray \cite{BresMur}, Buckmaster and Vicol \cite{BucVic},
Elgindi and Jeong \cite{ElgJeo}, and in the compressible case by 
Chiodaroli et al. \cite{ChKrMaSwI}, De Lellis and Sz\' ekelyhidi \cite{DelSze3}, among others. The class of solutions 
\tb{obtained through the vanishing viscosity limit may represent a chance} to restore well--posedness of the Euler system in some sense.

A rigorous identification of the vanishing viscosity limit is hampered by two principal 
difficulties:
\begin{itemize}
\item the absence of sufficiently strong uniform bounds to perform the limit in nonlinear terms; 
\item the viscous boundary layer created by the incompatibility of the no--slip boundary condition satisfied by the viscous fluid 
and the impermeability condition for the Euler system, see the survey by E \cite{E1}. 
\end{itemize}
The essential role of viscosity in a neighborhood of objects immersed in a fluid in motion is demonstrated by 
D'Alembert's paradox, see Stewartson \cite{Stew}. The interaction of the fluid with the physical boundary 
in the high Reynolds number regime gives rise to turbulent 
phenomena observed in real world experiments, see 
the \tb{monograph by Davidson \cite{DAVI}
or the nice survey by} Bonheure, Gazzola, and Sperone \cite{BoGaSp}, among many others.

Leaving apart the rather complex problem of viscous boundary layer, one may still legitimately ask if the Euler system can be used to describe the vanishing viscosity limit out of the boundary. The appearance of wakes, fluid separation and similar phenomena observed in the turbulent regime suggests to use \emph{statistical methods} to obtain a rigorous description. Apparently, the fluid is driven by the instabilities in the boundary area; whence one may anticipate the Euler system with complicated oscillatory boundary conditions 
to be the relevant model. The role of boundary conditions is often replaced by random forcing represented by a stochastic integral. The goal of this paper is to discuss to which extent such a scenario can be rigorously justified. 

We consider a compressible, viscous \emph{Navier--Stokes fluid} in an unbounded domain $Q \subset R^d$, $d=2,3$, exterior to a simply connected convex compact set $B$. We suppose the non-slip boundary conditions on $\partial B$, and prescribe the far field values of the density 
$\vr$, and the velocity $\vu$ for $|x| \to \infty$. The relevant system of equations reads:
\begin{equation} \label{i1}
\partial_t \vr + \Div (\vr \vu) = 0,
\end{equation}
\begin{equation} \label{i2}
\partial_t (\vr \vu) + \Div (\vr \vu \otimes \vu) + \Grad p(\vr) = \Div \mathbb{S} (\Grad \vu), 
\end{equation}
with the boundary and far field conditions 
\begin{equation} \label{i3}
\vu|_{\partial Q} = 0, 
\end{equation} 
\begin{equation} \label{i4}
\vu \to \vu_\infty, \ \vr \to \vr_\infty \ \mbox{as}\ |x| \to \infty.
\end{equation}
The viscous stress $\mathbb{S}$ is given by Newton's rheological law 
\begin{equation} \label{i5}
\mathbb{S} (\Grad \vu) = \mu \left( \Grad \vu + \Grad^t \vu - \frac{2}{d} \Div \vu \mathbb{I} \right) + 
\lambda \Div \vu \mathbb{I}, \ \mu > 0, \ \lambda \geq 0;
\end{equation}
the pressure $p$ is related to the density through the isentropic equation of state 
\begin{equation} \label{i6} 
p(\vr) = a \vr^\gamma, \ a > 0, \ \gamma > 1.
\end{equation}

Our aim is to study the asymptotic behavior of solutions to the problem \eqref{i1}--\eqref{i4} in the vanishing viscosity regime 
$\mu \searrow 0$, $\lambda \searrow 0$. To fix ideas, suppose 
\[
\mu_n = \ep_n \mu,\ \lambda_n = \ep_n \lambda ,\ \ep_n \to 0 \ \mbox{as}\ n \to \infty,  
\]
and denote the corresponding solution $(\vrn, \vmn)$, $\vmn = \vrn \vun$. Extending the functions $\vu_\infty$, $\vr_\infty$ inside $Q$, $\vu_\infty|_{\partial Q} = 0$, we introduce the relative energy 
\[
E\left(\vr, \vu \ \Big| \vr_\infty, \vu_\infty \right) = 
\frac{1}{2} \vr |\vu - \vu_\infty|^2 + P(\vr) - P'(\vr_\infty) (\vr - \vr_\infty) - P(\vr_\infty),\ 
P(\vr) = \frac{a}{\gamma - 1} \vr^\gamma.
\]  

Motivated by the numerical experiments of Elling \cite{ELLI2}, \cite{ELLI3}, we describe the vanishing viscosity limit 
in a statistical way. To this end, we introduce 
a suitable  trajectory space $\mathcal{T} $
in which the solutions $(\vrn, \vmn)$ live and \tb{denote} $\mathfrak{P}(\mathcal{T})$ the set of Borel probability measures on 
$\mathcal{T}$. The C\` esaro averages
\[
\mathcal{V}_N = \frac{1}{N} \sum_{n=1}^N \delta_{(\vrn, \vmn)} \in \mathfrak{P} (\mathcal{T}),\ 
\delta_h \ \mbox{-Dirac mass centered at}\ h,
\]
represent \emph{statistical solutions} of the Navier--Stokes system in the vanishing viscosity regime.
To perform the asymptotic limit $N \to \infty$, 
we suppose the  energy bound 
\begin{equation}\label{eq:9}
\frac{1}{N} \sum_{n=1}^N \left[ \sup_{t \in (0,T) } \intQ{ E\left(\vrn, \vun  \Big| \vr_\infty, \vu_\infty \right) } 
+ \ep_n \int_0^T \intQ{ \mathbb{S} (\Grad \vun) : \Grad \vun } \dt \right]  \leq \Ov{\mathcal{E}}
\end{equation}
uniformly for $N \to \infty$.  

Our programme consists of several steps.
First, we show that the family of measures $\mathcal{V}_N$ is tight; whence
\[
\mathcal{V}_N \to \mathcal{V} \ \mbox{narrowly in}\ \mathfrak{P}(\mathcal{T}),
\] 
at least for a suitable subsequence.
Then,  we denote by $(r,\vc{w})$ the canonical process on $\mathcal{T}$ (see Section~\ref{s:3.1} for more details) and  introduce the following notion of \emph{statistical equivalence}.

\begin{Definition}[Statistical equivalence] \label{iD1}

Let  $\mathcal{P}_{i}$, $i=1,2$, be  Borel probability measures on $\mathcal{T}$. We say that the two measures  are \emph{statistically equivalent} if the 
following holds:
\begin{itemize}
\item {\bf Equality of expectations of density and momentum.}
\begin{equation} \label{SE1}
\begin{aligned}
\mathbb{E}_{\mathcal{P}_{1}}\left[ \int_0^T \intQ{ r \varphi } \dt \right] &= \mathbb{E}_{{\mathcal{P}_{2}}}\left[ \int_0^T \intQ{ r \varphi } \dt  \right],\\ 
\mathbb{E}_{\mathcal{P}_{1}}\left[ \int_0^T \intQ{ \vc{w} \cdot \bfphi } \dt \right] &= \mathbb{E}_{{\mathcal{P}_{2}}}\left[ \int_0^T \intQ{ \vc{w} \cdot \bfphi } \dt \right],
\end{aligned}
\end{equation}
for any $\varphi \in \DC((0,T) \times Q)$, $\bfphi \in \DC((0,T) \times Q; R^d)$.
\item {\bf Equality of expectations of kinetic, internal and angular energy.}
\begin{equation} \label{SE2}
\begin{split}
\mathbb{E}_{\mathcal{P}_{1}}\left[ \int_0^T \intQ{ 1_{r > 0} \frac{|\vc{w} |^2}{r} \varphi      } \dt \right] &= 
\mathbb{E}_{\mathcal{P}_{2}}\left[ \int_0^T  \intQ{ 1_{r > 0} \frac{|\vc{w} |^2}{r} \varphi     }    \dt  \right], \\ 
\mathbb{E}_{\mathcal{P}_{1}}\left[ \int_0^T \intQ{ P(r) \varphi } \dt \right] &= 
\mathbb{E}_{{\mathcal{P}_{2}}}\left[ \int_0^T \intQ{ P(r) \varphi } \dt \right],\\ 
\mathbb{E}_{\mathcal{P}_{1}}\left[ \int_0^T \intQ{ 1_{r > 0} \frac{1}{r} (\mathbb{J}_{x_0} \cdot \vc{w} ) \cdot \vc{w} \ \varphi } \dt \right]
&= \mathbb{E}_{{\mathcal{P}_{2}}}\left[ \int_0^T \intQ{ 1_{r > 0} \frac{1}{r} (\mathbb{J}_{x_0} \cdot \vc{w} ) \cdot \vc{w} \ \varphi } \dt\right],
\end{split}
\end{equation}
\tb{for any} $x_0 \in R^d$, and any $\varphi \in \DC((0,T) \times Q)$, where 
\[
\mathbb{J}_{x_0}(x) = |x - x_0|^2 \mathbb{I} - (x - x_0) \otimes (x - x_0).
\]
\end{itemize}
In \eqref{SE1} and \eqref{SE2}, we tacitly assume that all the  expectations  are finite. In the same spirit, we say that two stochastic processes defined on possibly different probability spaces are statistically equivalent, if their probability laws are statistically equivalent.
\end{Definition}

\begin{Remark}\label{r:1.2}
The reader may have expected the statistical equivalence of two stochastic processes to be defined as their equivalence in law. We consider instead a much weaker notion, which postulates only equality of expectations of certain statistically and physically relevant quantities characterizing the fluid flow. In particular, expectations of other quantities or higher order moments do not need to coincide.
\end{Remark}

Next, we \emph{suppose} that \tb{the limit} $\mathcal{V}$ is statistically equivalent
to the probability law induced by a process $(\tvr, \tvm)$, which  is a solution of the stochastically driven Euler equation. Namely,  the process $(\tvr, \tvm)$ solves in the sense of distributions the system
\begin{equation} \label{i7}
\begin{split}
\D \tvr + \Div \tvm \dt &= 0, \\
\D \tvm + \Div \left( \frac{\tvm \otimes \tvm}{\tvr} \right) \dt + \Grad p(\tvr) \dt &= \vc{F} \D W,
\end{split}
\end{equation}
where 
\[
\vc{F} \D W = \sum_{k \geq 1} \vc{F}_k \D W_k
\]
with a family of diffusion coefficients $(\vc{F}_k )_{k \geq 1}$ satisfying a suitable stochastic integrability assumption and a cylindrical Wiener process $(W_k)_{k \geq 1}$.

\begin{Remark}\label{r:1.3}
It becomes clear at this point that a stronger definition of statistical equivalence, such as equality in law, would be too restrictive and would immediately rule out the possibility of modeling the inviscid limit by the stochastic Euler system. Indeed it will be seen in the analysis below that the limit measure $\mathcal{V}$ is supported on  momenta of bounded variation in time  with values in some negative Sobolev space. But this is not the case for $\tvm$ satisfying \eqref{i7} due to the limited regularity of the stochastic integral, unless  $\mathbf{F}\equiv 0$. However, the statistical equivalence in the sense Definition~\ref{iD1}, specifically, equality of only several chosen moments, is \tb{{\it a priori}} not excluded. 

Similarly, note that strengthening the statistical bound \eqref{eq:9} to 
$$
\sup_{n\in\mathbb{N}}\left[ \sup_{t \in (0,T) } \intQ{ E\left(\vrn, \vun  \Big| \vr_\infty, \vu_\infty \right) } 
+ \ep_n \int_0^T \intQ{ \mathbb{S} (\Grad \vun) : \Grad \vun } \dt \right]  \leq \Ov{\mathcal{E}}
$$
would imply an essentially \tb{uniform bound} for the limit process, which is also not compatible with solutions to the stochastic Euler system \eqref{i7}. This \tb{motivates} the more natural and \tb{essentially} weaker statistically uniform bound in the sense of the average in \eqref{eq:9}.
\end{Remark}

Let $(\vr,\vm)$ be a Skorokhod representation of  the limit measure $\mathcal{V}$. Our main result asserts that validity of \eqref{i7} necessarily implies: 
\begin{enumerate}
\item $(\vr, \vm)$ is a weak statistical solution of the \emph{deterministic} Euler system, \tb{in particular, 
the genuinely stochastic model becomes irrelevant.}
\item The sequence $(\vr_n, \vm_n)$ S--converges in the sense of \cite{Fei2020A} to a parametrized 
measure
\[
\widetilde{\mathcal{V}} \in L^\infty_{\rm weak-(*)}([0,T] \times Q; \mathfrak{P}[R^{d+1}]),
\] 
\[
\begin{split}
\int_0^T \intQ{ \left< \widetilde{\mathcal{V}}_{t,x}; b \right> \varphi } \dt &= 
\expe{ \int_0^T \intQ{ b(\vr, \vc{m}) \varphi} \dt } 
\\ \mbox{for any}\ b \in C_c(R^{d + 1}), \ &\mbox{\tb{and any}}\ \varphi \in C_c((0,T) \times Q).
\end{split}
\]

\item If, in addition, the barycenter 
\[
(\overline\vr, \overline\vm) = \int_{\mathcal{T}} (r, \vc{w}) \D \mathcal{V} \in \mathcal{T}
\]
is a weak solution to the Euler system, then the limit is deterministic
\[
\mathcal{V} = \delta_{(\overline\vr, \overline\vm)},
\] 
and
 the sequence $(\vrn, \vmn)$ statistically converges to $(\overline\vr, \overline\vm)$, specifically,
\begin{equation} \label{i8}
\frac{1}{N} \# \left\{ n \leq N \Big|  \| \vrn - \overline\vr \|_{L^\gamma(K)} + \| \vmn - 
\overline\vm \|_{L^{\frac{2 \gamma}{\gamma + 1}}(K; R^d)} > \ep  \right\} \to 
0 \ \mbox{as}\ N \to \infty
\end{equation}
for any $\ep > 0$, and any compact \tb{$K \subset [0,T] \times Q$}.
\end{enumerate}
The result can be extended to the driving force 
\[
\sigma \cdot \Grad \vm \circ \D W_1 + \vc{F} \D W_2 
\]
in the absence of the obstacle.

As a corollary, we obtain the dichotomy proved in \cite{MarEd} in the case $Q = R^d$: 
If  
\[
\vr_n \to \vr\ \mbox{weakly-(*) in}\ L^\infty(0,T; L^\gamma_{\rm loc} (Q)), \ \vm_n \to \vm 
\ \mbox{weakly-(*) in}\ L^\infty(0,T; L^{\frac{2 \gamma}{\gamma + 1}}_{\rm loc} (Q)), 
\]
then either 
\[
\vr_n \to \vr \ \mbox{in}\ L^\gamma_{\rm loc}([0,T] \times Q),\ 
\vm_n \to \vm \ \mbox{in}\ L^{\frac{2 \gamma}{\gamma + 1}}_{\rm loc}([0,T] \times Q; R^d),
\]
or $(\vr, \vm)$ \emph{is not} a solution of the Euler system \eqref{i7} (with $\mathbf{F} = 0$.)

To summarize, the above results can be interpreted in the following way. If we adopt the statistical limit as our working hypothesis, then 
the limit fluid motion is never statistically equivalent to a stochastic Euler system. If, in addition, we accept the Kolmogorov hypothesis advocated 
by Chen and Glimm \cite{CheGli} for compressible turbulence \tb{(in the case $Q = R^d$)}, then the S--convergence is the right tool to identify the limit. Last, if 
the expected value (barycenter) of the limit statistical solution solves the Euler system, then the statistical limit is a single deterministic solution. \tb{It is worth noting that the
results are independent of a specific choice of the initial data as well as the boundary conditions on $\partial Q$. This 
fact gives rise to full flexibility to cover all the physically relevant situations. In addition, we do not postulate any form 
of energy balance, neither on the approximate nor on the limit level. In the light of the experimental evidence, 
cf. \cite{BoGaSp}, our results strongly indicate that a stochastically driven Euler system 
is not a relevant model of compressible turbulence driven by a rigid obstacle. Indeed the graphic material collected 
in \cite{BoGaSp} is more reminiscent of the weak rather than strong convergence in the high Reynolds number regime, which, 
in view of the above results, 
excludes the Euler system to describe the statistical limit.}

The paper is organized as follows. In Section \ref{ns}, we collect the available results concerning the 
Navier--Stokes problem \eqref{i1}--\eqref{i6}. In Section \ref{sl}, we identify the measure $\mathcal{V}$ -- the statistical 
limit of the sequence $(\vrn, \vmn)_{n \geq 1}$. 
\tb{In Section \ref{R},
we associate to the statistical 
limit $\mathcal{V}$ a defect measure -- Reynolds stress tensor $\mathfrak{R}$ -- characterizing possible oscillations and concentrations 
created in the limit process.  It turns out that $\mathfrak{R}$ is a positive semi--definite tensor--valued finite measure 
on $Q$. Section \ref{sc} is the heart of the paper. We show that $\mathfrak{R}$ vanishes as long as the obstacle $B$ is a convex set, 
and then we examine the statistical convergence in the vanishing viscosity limit.} The paper is concluded by a short discussion in Section \ref{D}.

\section{Navier--Stokes system in exterior domain}
\label{ns}

We recall the available results for the Navier--Stokes system \eqref{i1}--\eqref{i6}. For the sake of simplicity, we suppose that 
$\partial Q$ is regular and $B = R^d \setminus Q$ is a simply connected compact set. In addition, we suppose that
\begin{equation} \label{ns1}
\vr_\infty > 0,\ \vu_\infty \in R \ \mbox{are given constant fields.}
\end{equation}
Accordingly, we may extend $\vu_\infty$ inside $Q$ in such a way that 
\begin{equation} \label{ns2}
\vu_\infty \in C^\infty(Q; R^d), \ \vu_\infty(x) = 0 \ \mbox{for}\ |x| < L,\ 
\vu_\infty(x) = \vu_\infty \ \mbox{for}\ |x| > 2 L 
\end{equation}
for some $L > 0$.

\subsection{Weak solutions} 

First we introduce finite energy \emph{weak solution} to the problem \eqref{i1}--\eqref{i6}. The total energy   
defined as  
\[
E\left(\vr, \vm \right) = 
\left\{ \begin{array}{l} \frac{1}{2} \frac{|\vm|^2}{\vr} + P(\vr) 
\ \mbox{if}\ \vr > 0, \\ 
0 \ \mbox{if}\ \vr = 0,\ \vm = 0, \\ \infty \ \mbox{otherwise}   \end{array} \right.  \ \vm = \vr \vu 
\]
is a convex l.s.c. function of $(\vr, \vm) \in R^{d + 1}$. In view of the non--zero far--field conditions, it is more convenient 
to consider the relative energy 
\[
E\left(\vr, \vm \Big| \vr_\infty, \vu_\infty \right) = 
\frac{1}{2} \frac{|\vm|^2}{\vr} - \vm \cdot \vu_\infty + \frac{1}{2} \vr |\vu_\infty|^2 + 
P(\vr) - P'(\vr_\infty) (\vr - \vr_\infty) - P(\vr_\infty).
\]
As $E$ is convex, the relative energy
can be interpreted as the so--called Bregman distance between $(\vr, \vm)$ and $(\vr_\infty, \vm_\infty)$, $\vm_\infty = \vr_\infty \vu_\infty$,  
specifically 
\[
E\left(\vr, \vm \Big| \vr_\infty, \vu_\infty \right) = E(\vr, \vm) - \partial_{\vr, \vm} E(\vr_\infty, \vm_\infty) 
\cdot (\vr - \vr_\infty, \vm - \vm_\infty) - E(\vr_\infty, \vm_\infty).
\] 

\begin{Definition}[Weak solution to Navier--Stokes system] \label{Dns1}
We say that $(\vr, \vu)$ is \emph{finite energy weak solution} to the Navier--Stokes system \eqref{i1}--\eqref{i6} 
in $(0,T) \times Q$ 
if:  
\begin{itemize}
\item {\bf Finite energy and dissipation rate} 
\begin{equation} \label{ns3}
\intQ{ E\Big(\vr, \vm \Big| \vr_\infty, \vu_\infty \Big) (\tau, \cdot) } + \int_0^\tau \intQ{ \mathbb{S}(\Grad \vu) : \Grad \vu } \dt 
< \infty
\end{equation}
holds for any $0 < \tau \leq T$; 

\item {\bf Equation of continuity}
\[
\left[ \intQ{ \vr \varphi } \right]_{t=\tau_1}^{t= \tau_2} = 
\int_{\tau_1}^{\tau_2} \intQ{ \Big[ \vr \partial_t \varphi + \vr \vu \cdot \Grad \varphi \Big] } \dt
\]
holds for any $0 < \tau_1 < \tau_2 < T$ and $\varphi \in C^1_c ((0,T) \times Q)$;

\item
{\bf Momentum equation} 
\[
\begin{split}
&\left[ \intQ{ \vr \vu \cdot \bfphi } \right]_{ t = \tau_1 }^{t = \tau_2} \\ &= 
\int_{\tau_1}^{\tau_2} \intQ{ \Big[ \vr \vu \cdot \partial_t \bfphi + 
(\vr \vu \otimes \vu) : \Grad \bfphi + p(\vr) \Div \bfphi - \mathbb{S} (\Grad \vu) : \Grad \bfphi \Big] } \dt
\end{split}
\]
holds for any $0 < \tau_1 < \tau_2 < T$ and $\bfphi \in C^1_c ((0,T) \times Q; R^d)$.
                                                                                                                                                                                                                                                                                                                                                                                                                           \end{itemize}

\end{Definition}

Note that boundedness of the total energy and the dissipation rate stated in \eqref{ns3} yields 
the natural bound $\vr \geq 0$ and renders 
all integrals in the weak formulation finite. The definition is usually supplemented with the renormalized equation of continuity as well as the compatibility condition 
\begin{equation} \label{ns4}
(\vu - \vu_\infty) \in L^2(0,T; D^{1,2}(Q; R^d) )
\end{equation}
where $D^{1,2}$ denotes the homogeneous Sobolev space. The energy bound \eqref{ns3} then follows from the associated energy 
inequality 
\[
\begin{split}
\frac{\D}{ \dt} & \intQ{ E \left( \vr, \vu \Big| \vr_\infty, \vu_\infty \right) } + \intQ{ \mathbb{S}(\Grad \vu) : \Grad \vu } \\ 
& \leq -\intQ{ \Big( \vr \vu \otimes \vu + p(\vr) \mathbb{I} \Big) : \Grad \vu_\infty} + 
\frac{1}{2} \intQ{ \vr \vu \cdot \Grad |\vu_\infty|^2 } + \intQ{ \mathbb{S}(\Grad \vu) : \Grad \vu_\infty }.
\end{split}
\]
We refer the reader to the monographs of Lions \cite{LI4}, 
Novotn\' y and Stra\v skraba \cite[Section 7.12.6]{NOST} or Kra\v cmar, Ne\v casov\' a, and Novotn\' y \cite{KrNeNo} 
for the relevant existence results for the initial--value problem under various restrictions imposed on the adiabatic coefficient 
$\gamma$.

Here we consider families of solutions that may fail to satisfy \eqref{ns4}, meaning we do not really specify the boundary behavior of the velocity on the rigid body. Similarly, we also ignore 
the behavior of the initial state.
The only piece of information necessary for our analysis is the uniform bound \eqref{ns3}. 

\section{Statistical limit}
\label{sl}

For a given sequence of $\ep_n \to 0$ we consider the vanishing viscosity limit 
\[
\mu_n = \ep_n \mu, \ \lambda_n = \ep_n \lambda , \ \mu > 0, \ \lambda \geq 0.
\]
We suppose that the Navier--Stokes system admits a related sequence of weak solutions $(\vr_n, \vu_n)_{n=1}^\infty$ in the sense of Definition \ref{Dns1}. Our goal is to study the \emph{statistical properties} of  $(\vr_n, \vu_n)_{n=1}^\infty$. To this end, we associate 
to this sequence a family of measures $\mathcal{V}_N$ supported on the \emph{trajectory space} 
\[
\mathcal{T} = C_{{\rm weak}}([0,T]; L^\gamma_{\rm loc}(Q) \times L^{\frac{2 \gamma}{\gamma + 1}}_{\rm loc}(Q; R^d)),  
\]
\begin{equation} \label{sl1}
\mathcal{V}_N = \frac{1}{N} \sum_{n=1}^N \delta_{(\vr_n, \vm_n)},\ \vm_n = \vr_n \vu_n,
\end{equation}
where $\delta$ denotes the Dirac mass. Note that any finite energy weak solution belongs to $\mathcal{T}$.
Moreover, motivated by the energy bound \eqref{ns3}, we assume
\begin{equation} \label{sl4}
\frac{1}{N} \sum_{n=1}^N \left[ \sup_{0 \leq \tau \leq T}\intQ{ E\left(\vr_n, \vm_n \Big| \vr_\infty, \vu_\infty \right) (\tau, \cdot) } + 
\ep_n \int_0^T \intQ{ \mathbb{S}(\Grad \vu_n) : \Grad \vu_n } \dt \right] \leq \Ov{\mathcal{E}}
\end{equation}
uniformly for $N \to \infty$. Note in particular that we do not assume a uniform bound of the form
\begin{equation*}
\sup_{n}\left[ \sup_{0 \leq \tau \leq T}\intQ{ E\left(\vr_n, \vm_n \Big| \vr_\infty, \vu_\infty \right) (\tau, \cdot) } + 
\ep_n \int_0^T \intQ{ \mathbb{S}(\Grad \vu_n) : \Grad \vu_n } \dt \right] <\infty.
\end{equation*}
In fact, such an assumption would render parts of our analysis below trivial. We consider instead the weaker statistically uniform bound in the sense of the average in \eqref{sl4}, which is more natural for the problem under consideration.

\subsection{Statistical convergence}\label{s:3.1}

Our goal is to show 
\begin{equation} \label{sl5}
\mathcal{V}_N \to \mathcal{V} 
\ \mbox{narrowly in} \ \mathfrak{P}[ \mathcal{T} ] 
\end{equation}
at least for a suitable subsequence $N_k \to \infty$ as $k \to \infty$. According to Prokhorov theorem it is enough to show that 
$(\mathcal{V}_{N})_{N \geq 1}$ is tight. To this end, we denote by $(r,\vc{w})$  the canonical process on $\mathcal{T}$, that is,
$$
(r,\vc{w}):\mathcal{T}\to \mathcal{T},\  (r,\vc{w})(\omega,t)=\omega(t)\ \mbox{for}\ \omega\in\mathcal{T},\, t\in[0,T].
$$
We note that the velocity $\vv$ satisfying $\vc{w}=r\vv$ is well defined  under each $\mathcal{V}_{N}$.
For a probability measure $\mathcal{Q}\in\mathfrak{P}(\mathcal{T})$, we denote by $\mathbb{E}_{\mathcal{Q}}$ the expected value in the probability space $(\mathcal{T},\mathfrak{B}[\mathcal{T}],\mathcal{Q})$. From Section~\ref{R} on, we also use the notation $\mathbb{E}$ without any subscript to denote the expectation on the standard probability space $(\Omega,\mathfrak{B},\mathcal{P})$.

\begin{Lemma} \label{Lsl1}

Under the hypothesis \eqref{sl4}, the family $(\mathcal{V}_{N})_{N \geq 1}$ is tight in $\mathfrak{P}[\mathcal{T}]$.

\end{Lemma}

\begin{proof}
First observe, by virtue of Young's inequality, 
\begin{equation} \label{sl6a}
|\vc{w}|^{\frac{2 \gamma}{\gamma + 1}} + r^\gamma \aleq E(r, \vc{w}).
\end{equation}
Consequently, it follows from \eqref{sl4} that
\begin{equation} \label{sl6}
\begin{aligned}
&\mathbb{E}_{\mathcal{V}_N} \left[ \sup_{0 \leq \tau \leq T} \int_K |\vc{w}|^{\frac{2 \gamma}{\gamma + 1}}\ \dx + \sup_{0 \leq \tau \leq T} \int_K r^\gamma \ \dx 
\right]\\
 &\qquad = \frac{1}{N}\sum_{n=1}^{N}  \left[ \sup_{0 \leq \tau \leq T} \int_K |\vm_{n}|^{\frac{2 \gamma}{\gamma + 1}}\ \dx + \sup_{0 \leq \tau \leq T} \int_K \vr_{n}^\gamma \ \dx 
\right]\leq c (K, \Ov{\mathcal{E}}) 
\end{aligned}
\end{equation}
for any compact $K \subset \Ov{Q}$. 

To complete the proof, it is enough to show uniform bounds on the modulus of continuity of processes
\[
t \in [0,T] \mapsto \intQ{ r(t) \varphi },\ 
t \in [0,T] \mapsto \intQ{ \vc{w}(t) \cdot \bfphi },\ \varphi \in C^1_c(Q),\ \bfphi \in C^1_c (Q; R^d),
\]
under $\mathcal{V}_{N}$.
As $\vr_n$, $\vm_n$ satisfy the equation of continuity, we deduce from \eqref{sl6}
\begin{equation} \label{sl7}
\begin{aligned}
&\mathbb{E}_{\mathcal{V}_N} \left[ \sup_{0\leq\tau_{1}<\tau_{2}\leq T} \frac{\left| \intQ{ \Big[ r(\tau_2, \cdot) - r(\tau_1, \cdot) \Big] \varphi  } \right| }{|\tau_{2}-\tau_{1}|}\right] 
\leq c(\varphi, \Ov{\mathcal{E}}). 
\end{aligned}
\end{equation}
Similarly, the momentum equation yields
\begin{equation} \label{sl8}
\mathbb{E}_{\mathcal{V}_N} \left[\sup_{0\leq\tau_{1}<\tau_{2}\leq T} \frac{\left| \intQ{ \Big[ \vc{w}(\tau_2, \cdot) - \vc{w}(\tau_1, \cdot) \Big] \cdot \bfphi } \right| }{|\tau_{2}-\tau_{1}|^{\frac12}}\right] \leq c(\bfphi, \Ov{\mathcal{E}}). 
\end{equation}
Here, we have used
\begin{equation} \label{sl9}
\|\sqrt{\ep_{n}} \mathbb{S}(\Grad \vu_n) \|^{2}_{L^2(Q; R^{d \times d})} \lesssim \ep_{n}\max\{\mu,\lambda\}  \intQ{ \mathbb{S}(\Grad \vu_n) 
: \Grad \vu_n }.
\end{equation}

\end{proof}

As  the weak topology is not metrizable, the space $\mathcal{T}$ is not a Polish space
and Prokhorov theorem does not directly apply. However, $\mathcal{T}$ is a sub-Polish space in the sense of \cite[Definition~2.1.3]{BrFeHobook}. Therefore, the Jakubowski--Skorokhod \cite{Jakub} implies \eqref{sl5}, at least for a suitable subsequence. 
In particular,
\begin{equation} \label{pro2}
\begin{split}
&\frac{1}{N_{k}} \sum_{n=1}^{N_k} b \left(\int_{0}^{T}\intQ{ \vr_n \varphi_1} \dt, \dots, \int_{0}^{T}\intQ{ \vr_n \varphi_{m_{1}}}\dt, \right. \\
&\hspace{3cm}
\left. \int_{0}^{T}\intQ{ \vm_n \cdot \bfphi_1}\dt , \dots, \int_{0}^{T}\intQ{ \vm_n \cdot \bfphi_{m_{2}}}  \dt \right) \\ 
&\to \mathbb{E}_\mathcal{V}
\left[ b \left(\intQ{ r \varphi_1} , \dots, \intQ{ r \varphi_{m_{1}}}, 
\intQ{ \vc{w} \cdot \bfphi_1} , \dots, \intQ{ \vc{w} \cdot \bfphi_{m_{2}}}   \right) \right] \ \mbox{as}\ k \to \infty 
\end{split} 
\end{equation}
for any $\varphi_i \in C^1_c(Q)$, $i=1,\dots, m_{1}$, $\bfphi_i \in C^1_c(Q; R^d)$, $i=1,\dots, m_{2}$,  $m_{1},m_{2}\in\mathbb{N}$, and any 
$b \in C_c \left( R^{m_{1}+m_{2}} \right)$. In addition, we may also define the barycenter of $\mathcal{V}$ on the trajectory space,
\[
\begin{split}
(\overline\vr, \overline\vm) &= \mathbb{E}_{\mathcal{V}}\left[ (r, \vc{w}) \right] ,\\
\int_0^T \intQ{ \overline\vr \varphi } \dt = 
\mathbb{E}_{\mathcal{V}} \left[ \int_0^T \intQ{ r \varphi } \dt \right] &= \lim_{k \to \infty} \frac{1}{N_k} \sum_{n=1}^{N_k} 
\left[ \int_0^T \intQ{ \vrn \varphi } \dt \right]
\\
 \int_0^T \intQ{ \overline\vm \cdot \bfphi } \dt = 
\mathbb{E}_{\mathcal{V}} \left[ \int_0^T \intQ{ \vc{w} \cdot \bfphi } \dt \right] &= \lim_{k \to \infty} \frac{1}{N_k} \sum_{n=1}^{N_k} 
\left[ \int_0^T \intQ{ \vmn \cdot\bfphi } \dt \right]\\ 
\mbox{for any}\ \varphi \in C^1_c((0,T) \times Q), \ \bfphi \in C^1_c ((0,T) \times Q; R^d).
\end{split}
\]

Summarizing we obtain: 

\begin{Proposition}[Statistical vanishing viscosity limit] \label{Psl1}

Let $(\vrn, \vu_n)_{n=1}^\infty$ be a sequence of weak solutions to the Navier--Stokes system in the sense of Definition 
\ref{Dns1} with the viscosity coefficients 
\[
\mu_n = \ep_n \mu, \ \lambda_n = \ep_n \lambda , \ \mu > 0, \ \lambda \geq 0, \ \ep_n \!\searrow 0. 
\]
Suppose the C\` esaro averages of the total (relative) energy are bounded as in \eqref{sl4}. 

Then there exists $N_k \to \infty$ and a probability measure $\mathcal{V} \in \mathfrak{P}[\mathcal{T}]$ on the trajectory space $\mathcal{T}$ such that  \eqref{pro2} holds.

\end{Proposition}

\section{Reynolds defect}
\label{R}

In order to identify the statistical limit $\mathcal{V}$ with a stochastic process, we apply Skorokhod representation theorem, or rather its generalized version by Jakubowski \cite{Jakub}. It is convenient to extend the class of 
variables $(\vr_n, \vm_n)$ to their nonlinear composition appearing in both the relative energy and the momentum equation. 
Specifically, we consider the quantities 
\begin{equation}\label{eq:100}
1_{\vr_n > 0} \frac{\vmn \otimes \vmn}{\vrn} \in L^\infty_{\rm weak-(*)} (0,T; \mathcal{M}(Q; R^{d \times d}_{\rm sym} ))  ,\ \mbox{and}\ p(\vr_n) \in L^\infty_{\rm weak-(*)} (0,T; \mathcal{M}(Q)), 
\end{equation}
where the symbol $\mathcal{M}$ denotes the space of all Radon (not necessarily finite) measures on $Q$. In view of the fact that
$$
L^\infty_{\rm weak-(*)} (0,T; \mathcal{M}(Q; R^{d \times d}_{\rm sym} )) = \big( L^{1}(0,T; C_{c}(Q,R^{d\times d}_{\rm sym}))\big)^{*},
$$
where  $C_{c}(Q,R_{\rm sym}^{ d\times d})$ is separable, the function spaces in \eqref{eq:100} satisfy the assumptions of Jakubowski's theorem \cite{Jakub}.
 We consider the extended  
measure
\[
\begin{split}
&\Ov{\mathcal{V}}_N = \frac{1}{N} \sum_{n=1}^N \delta_{ \left[
(\vr_n, \vm_n), \ 1_{\vr_n > 0} \frac{\vmn \otimes \vmn}{\vrn},\ p(\vr_n) \right] }\\ 
&\in \mathfrak{P} \left[\mathcal{T} \times 
L^\infty_{\rm weak-(*)} (0,T; \mathcal{M}(Q; R^{d \times d}_{\rm sym} )) \times L^\infty_{\rm weak-(*)} (0,T; \mathcal{M}(Q)) \right].
\end{split}
\]

It follows from the energy bound \eqref{sl4} and Lemma~\ref{Lsl1} that the family $(\Ov{\mathcal{V}}_N)_{N  \geq 0}$ is tight. Now, we apply 
the version of Skorokhod representation theorem due to Jakubowski \cite{Jakub}. Hence there is a subsequence $N_{k}\to\infty$ and a family of random variables 
\[
(\tvr_k, \tvm_k) \in \mathcal{T},\ 1_{\tvr_k > 0} \frac{\tvm_k \otimes \tvm_k}{\tvr_k} \in 
L^\infty_{{\rm weak - (*)}}(0,T; \mathcal{M}(Q; R^{d \times d}_{\rm sym})), \ p(\tvr_k) \in 
L^\infty_{{\rm weak - (*)}}(0,T; \mathcal{M}(Q)) 
\]  
defined on the standard probability space $(\Omega, \mathfrak{B}, \mathcal{P})$ with the law $\Ov{\mathcal{V}}_{N_{k}}$ 
and such that 
\begin{equation} \label{sl10a}
\begin{split}
(\tvr_k, \tvm_k) &\to (\vr, \vm) \ \mbox{in}\ \mathcal{T} \ \prst - \mbox{a.s.}, \\
1_{\tvr_k > 0} \frac{\tvm_k \otimes \tvm_k}{\tvr_k} &\to 
\Ov{ \left[1_{\vr > 0} \frac{\vm \otimes \vm}{\vr} \right] } \ \mbox{weakly-(*) in}\ 
L^\infty_{{\rm weak - (*)}}(0,T; \mathcal{M}(Q; R^{d \times d}_{\rm sym}))  \ \prst - \mbox{a.s.}, \\ 
p(\tvr_k) &\to 
\Ov{p(\vr)} \ \mbox{weakly-(*) in}\ 
L^\infty_{{\rm weak - (*)}}(0,T; \mathcal{M}(Q))  \ \prst - \mbox{a.s.}
\end{split}
\end{equation}

\subsection{Asymptotic limit in the equation of continuity}

In view of \eqref{sl5}, 
\begin{equation} \label{sl10}
\mathcal{V} = \mathcal{L}_{\mathcal{T}}[\vr, \vm],
\end{equation}
where $\mathcal{V}$ is the statistical limit identified in Proposition \ref{Psl1}.
Moreover,
as the law of the new variables coincides with $\Ov{\mathcal{V}}_{N_{k}}$, the continuity equation is satisfied $\prst-$a.s.: 
\begin{equation} \label{sl12bis}
\left[ \intQ{ \tvr_k \varphi } \right]_{t=\tau_1}^{t= \tau_2} = 
\int_{\tau_1}^{\tau_2} \intQ{ \Big[ \tvr_k \partial_t \varphi + \tvm_k \cdot \Grad \varphi \Big] } \dt
\end{equation}
for any $0 < \tau_1 < \tau_2 < T$ and $\varphi \in C^1_c ((0,T) \times Q)$. Thus, letting $k \to \infty$, we get
\begin{equation} \label{sl12}
\left[ \intQ{ \vr \varphi } \right]_{t=\tau_1}^{t= \tau_2} = 
\int_{\tau_1}^{\tau_2} \intQ{ \Big[ \vr \partial_t \varphi + \vm \cdot \Grad \varphi \Big] } \dt
\end{equation}
for any $0 < \tau_1 < \tau_2 < T$ and $\varphi \in C^1_c ((0,T) \times Q)$ $\prst-$a.s. 

\subsection{Asymptotic limit in the momentum equation}

In accordance with \eqref{sl4} we have 
\[
\expe{ \sup_{0 \leq \tau \leq T} \intQ{ E\left(\tvr_k, \tvm_k \Big| \vr_\infty, \vu_\infty \right)(\tau, \cdot) } }
\leq \Ov{\mathcal{E}}.
\]
Here and in the sequel, we denote by $\mathbb{E}$ (i.e. without any subscript) the expectation on the probability space $(\Omega,\mathfrak{B},\mathcal{P})$.
Consequently, it follows from \eqref{sl10a} that 
\begin{equation} \label{R1} 
E\left(\tvr_k, \tvm_k \Big| \vr_\infty, \vu_\infty \right) \to 
\Ov{E\left(\vr, \vm \Big| \vr_\infty, \vu_\infty \right)}\ \mbox{weakly-(*) in}\ 
L^\infty_{\rm weak - (*)}(0,T; \mathcal{M} (Q) ), 
\end{equation}
or, more specifically, 
\[
\int_0^T \psi \intQ{ E\left(\tvr_k, \tvm_k \Big| \vr_\infty, \vu_\infty \right) \varphi } \dt
\to \int_0^T \psi \int_Q \varphi \ \D \Ov{E\left(\vr, \vm \Big| \vr_\infty, \vu_\infty \right)}  \dt
\]
for any $\psi \in L^1(0,T), \ \varphi \in C_c(Q)$ $\prst-$a.s. Moreover, passing to expectations we get, by Fatou's lemma, 
\begin{equation} \label{R2}
\begin{split}
\expe{\int_0^T \psi \int_Q \varphi \ \D \Ov{E\left(\vr, \vm \Big| \vr_\infty, \vu_\infty \right)}  \dt}
&\leq \liminf_{k \to \infty} \expe{ \int_0^T \psi \intQ{ E\left(\tvr_k, \tvm_k \Big| \vr_\infty, \vu_\infty \right) \varphi } \dt}\\
&\leq \Ov{\mathcal{E}} \| \psi \|_{L^1(0,T)} \| \varphi \|_{C(Q)}, \ \psi, \varphi \geq 0.
\end{split}
\end{equation}
Consequently, unlike its components in \eqref{sl10a}, the limit relative energy is a \emph{finite} measure on $(0,T) \times Q$.

Finally, since $(\tvr_{k},\tvm_{k})$ has the law $\mathcal{V}_{N_{k}}$, it satisfies the corresponding momentum equation
$$
\begin{aligned}
\left[ \int_{Q}\tvm_{k}\cdot\bfphi\ \dx\right]_{t=\tau_{1}}^{t=\tau_{2}}=&\int_{\tau_{1}}^{\tau_{2}} \intQ{ \left[ \tvm_k \cdot \partial_t \bfphi + 1_{\tvr_k > 0}
\frac{ \tvm_k \otimes \tvm_k }{\tvr_k} : \Grad \bfphi + p(\tvr_k) \Div \bfphi \right] } \dt \\
&\quad-\sum_{n=1}^{N_{k}}1_{(\tvr_{k},\tvm_{k})=(\vr_{n},\vm_{n})}\ep_{n}\int_{\tau_{1}}^{\tau_{2}} \intQ{\Big[\mathbb{S}(\Grad \vu_{n}):\Grad\bfphi \Big]} \dt
\end{aligned}
$$
for any $0 < \tau_1 < \tau_2 < T$ and $\bfphi \in C^1_c ((0,T) \times Q; R^d)$. We shall prove that the right hand side vanishes as $k\to\infty$. In view of \eqref{sl9} and \eqref{sl4}, we obtain
$$
\frac{1}{N_{k}}\sum_{n=1}^{N_{k}}\left[ \int_{0}^{T}\|\sqrt{\ep_{n}}\mathbb{S}(\Grad \vu_{n})\|_{L^{2}(Q;R^{d\times d})}^{2}\ \dt\right
]\lesssim \overline{\mathcal{E}},
$$
hence
\begin{equation*}
\begin{aligned}
&\mathbb{E}\left[\left| \sum_{n=1}^{N_{k}}1_{(\tvr_{k},\tvm_{k})=(\vr_{n},\vm_{n})}\ep_{n}\int_{0}^{T} \intQ{\Big[\mathbb{S}(\Grad \vu_{n}):\Grad\bfphi \Big]} \dt\right|\right]\\
&\qquad =\mathbb{E}_{\mathcal{V}_{N_{k}}}\left[\left| \sum_{n=1}^{N_{k}}1_{(r,\vc{w})=(\vr_{n},\vm_{n})}\ep_{n}\int_{0}^{T} \intQ{\Big[\mathbb{S}(\Grad \vu_{n}):\Grad\bfphi \Big]} \dt\right|\right]\\
&\qquad =\frac{1}{N_{k}}\sum_{i=1}^{N_{k}}\left[\left| \sum_{n=1}^{N_{k}}1_{(\vr_{i},\vm_{i})=(\vr_{n},\vm_{n})}\ep_{n}\int_{0}^{T} \intQ{\Big[\mathbb{S}(\Grad \vu_{n}):\Grad\bfphi \Big]} \dt\right|\right]\\
&\qquad =\frac{1}{N_{k}}\sum_{i=1}^{N_{k}}\left[\left| \ep_{i}\int_{0}^{T} \intQ{\Big[\mathbb{S}(\Grad \vu_{i}):\Grad\bfphi \Big]} \dt\right|\right]\\
&\qquad \lesssim \|\bfphi\|_{C^{1}((0,T)\times Q)}\left(\frac{1}{N_{k}}\sum_{i=1}^{N_{k}}\ep_{i}\right)^{\frac12}\left(\frac{1}{N_{k}}\sum_{i=1}^{N_{k}} \left[\int_{0}^{T} \|\sqrt{\ep_{i}}\mathbb{S}(\Grad \vu_{i})\|_{L^{2}(Q;R^{d\times d})}^{2}  \dt\right]\right)^{\frac12}\\
&\qquad \lesssim \sqrt{\overline{\mathcal{E}}}\|\bfphi\|_{C^{1}((0,T)\times Q)}\left(\frac{1}{N_{k}}\sum_{i=1}^{N_{k}}\ep_{i}\right)^{\frac12}.
\end{aligned}
\end{equation*}
Since the above right hand side vanishes as $k\to\infty$, we conclude 
\begin{equation} \label{R4}
\int_0^T \left( \intQ{  \vm \cdot \partial_t \bfphi }  + \int_Q \Grad \bfphi : \D \ \Ov{ \left[ 1_{\vr > 0}
\frac{ \vm \otimes \vm }{\vr}  + p(\vr) \mathbb{I} \right]  } \right)  \dt = 0
\ \mbox{for any}\ \bfphi \in C^1_c((0,T) \times Q)
\end{equation}
$\prst-$a.s.

\subsection{Defect measure in the momentum equation}

Following \cite{MarEd} we rewrite equation \eqref{R4} in the form 
\[
\begin{split}
\int_0^T &\intQ{ \left[ \vm \cdot \partial_t \bfphi + 
1_{\vr > 0} \frac{ \vm \otimes \vm }{\vr}: \Grad \bfphi  + p(\vr) \Div \bfphi \right] } \dt \\ &= 
- \int_0^T \int_Q \Grad \bfphi : \left( \D \Ov{ \left[
1_{\vr > 0} \frac{ \vm \otimes \vm }{\vr}  + p(\vr) \mathbb{I} \right]  } - \left[
1_{\vr > 0} \frac{ \vm \otimes \vm }{\vr}  + p(\vr) \mathbb{I} \right] \dx
\right) \dt.
\end{split}
\]
The quantity 
\[
\mathfrak{R} =
\Ov{ \left[ 1_{\vr > 0} \frac{ \vm \otimes \vm }{\vr}  + p(\vr) \mathbb{I} \right]  }  - \left[
1_{\vr > 0} \frac{ \vm \otimes \vm }{\vr}  + p(\vr) \mathbb{I} \right] 
\in \ L^\infty_{\rm weak-(*)} (0,T; \mathcal{M}(Q; R^{d \times d}_{\rm sym})) 
\]
is called \emph{Reynolds defect}. A simple but crucial observation is that the tensor--valued measure $\mathfrak{R}$ 
is positively semi--definite in the sense that
\[
\int_0^T \psi \intQ{
\mathfrak{R} : (\xi \otimes \xi) \varphi } \dt \geq 0 \ \mbox{for any}\ \xi \in R^d 
\ \mbox{and any}\ \psi \in L^1(0,T), \ \varphi \in C_c(Q), \ \psi, \varphi \geq 0 \ \prst -\mbox{a.s.}
\] 
Indeed this follows from convexity of the function
\[
(\vr, \vm) \mapsto \frac{|\vm \cdot \xi|^2}{\vr} + p(\vr)|\xi|^2 \ \mbox{for any}\ \xi \in R^d,
\]
cf. \cite{MarEd} for details.

Our final goal in this section is to show that all components of $\mathfrak{R}$ are finite measures on $Q$. To see this, we compute 
its trace
\[
0 \leq {\rm trace}[\mathfrak{R}] = \Ov{ \frac{|\vm|^2}{\tvr} + d p(\vr) } - \left[ \frac{|\vm|^2}{\vr} + d p(\vr) \right], 
\]
and compare it with the defect of relative energy. More precisely, we first observe
\begin{equation} \label{R5}
\begin{split}
&\Ov{ \frac{1}{2} \frac{|\vm|^2 }{\vr} + P(\vr) } - \left[ \frac{1}{2} \frac{|\vm|^2 }{\vr} + P(\vr) \right]
\leq \max \left\{ \frac{1}{2}, \frac{1}{(\gamma - 1)d} \right\} \left( \Ov{ \frac{|\vm|^2}{\vr} + d p(\vr) } - \left[ \frac{|\vm|^2}{\vr} + d p(\vr) \right] \right)\\
&\qquad\leq \max \left\{ \frac{1}{2}, \frac{1}{(\gamma - 1)d} \right\} 
\max \left\{ 2 , (\gamma - 1)d \right\} \left( \Ov{ \frac{1}{2} \frac{|\vm|^2 }{\vr} + P(\vr) } - \left[ \frac{1}{2} \frac{|\vm|^2 }{\vr} + P(\vr) \right] \right).
\end{split}
\end{equation}
Next, we write 
\begin{equation} \label{R5a}
\begin{split}
&\int_0^T \psi \int_Q \varphi \left( \D \Ov{ \left[ \frac{1}{2} \frac{|\vm|^2 }{\vr} + P(\vr) \right]} - \left[ \frac{1}{2} \frac{|\vm|^2 }{\vr} + P(\vr) \right]\dx \right)\dt \\
 &=
\lim_{k \to \infty} \int_0^T \psi \intQ{ \varphi \left( \left[ \frac{1}{2} \frac{|\tvm_k|^2 }{\tvr_k} + P(\tvr_k) \right]
- \left[ \frac{1}{2} \frac{|\vm|^2 }{\vr} + P(\vr) \right] \right) } \dt \\ 
&= \lim_{k \to \infty} \int_0^T \psi \intQ{ \varphi \left( \frac{1}{2} \frac{|\tvm_k|^2 }{\tvr_k} + 2 \tvm_k \cdot \vu_\infty - 
\frac{1}{2} {\tvr_{k}} |\vu_\infty|^2 
+ P(\tvr_k) - 
P'(\vr_\infty)(\tvr_k - \vr_\infty) - P(\vr_\infty) \right) \!\!} \dt\\ 
&\quad- \int_0^T \psi \intQ{ \varphi \left( \frac{1}{2} \frac{|\vm|^2 }{\vr} + 2 \vm \cdot \vu_\infty - 
\frac{1}{2} {\vr} |\vu_\infty|^2 
+ P(\vr) - 
P'(\vr_\infty)(\vr - \vr_\infty) - P(\vr_\infty) \right) } \dt\\ 
&= \lim_{k \to \infty} \int_0^T \psi \intQ{ \varphi E \left( \tvr_k, \tvm_k \Big| \vr_\infty, \vm_{\infty} \right) } \dt
- \int_0^T \psi \intQ{ \varphi E \left( \vr, \vm \Big| \vr_\infty, \vm_{\infty} \right) } \dt \\ 
&= \int_0^T \psi \int_Q \varphi \ \D  \Ov{ E \left( \vr, \vm \Big| \vr_\infty, \vm_{\infty} \right) }  \dt
- \int_0^T \psi \intQ{ \varphi E \left( \vr, \vm \Big| \vr_\infty, \vm_{\infty} \right) } \dt
\end{split}
\end{equation}
for any $\psi \in L^1(0,T)$, $\varphi \in C_c(Q)$ $\prst-$a.s., where we have used 
\[
\int_0^T \psi \intQ{ \varphi \tvr_k } \dt \to 
\int_0^T \psi \intQ{ \varphi \vr } \dt,\ \int_0^T \psi \intQ{ \varphi \tvm_{k} \cdot \vu_\infty } \dt \to 
\int_0^T \psi \intQ{ \varphi \vm \cdot \vu_\infty } \dt
\] 
$\prst$--a.s. 

Thus, combining \eqref{R2} with \eqref{R5} we obtain the desired conclusion
\begin{equation} \label{R6}
\expe{ \int_0^T \psi \left( \int_Q \varphi \ \D\, {\rm trace}[\mathfrak{R}] \right) \dt } \leq c \Ov{\mathcal{E}} \| \psi \|_{L^1(0,T)} 
\| \varphi \|_{C(Q)}, \ \psi \in L^1(0,T),\ \varphi \in C_c(Q).
\end{equation}

\section{Limit problem}
\label{sc}

In the preceding two sections, we have identified the limit problem in the vanishing viscosity regime as a stochastic  process 
$(\vr, \vm)$ on the probability space $(\Omega, \mathfrak{B}, \prst)$, with paths in the trajectory space $\mathcal{T}$, satisfying $\prst-$a.s.
\begin{equation} \label{sc1}
\int_0^T \intQ{ \Big[ \vr \partial_t \varphi + \vm \cdot \Grad \varphi \Big] } \dt = 0 \ \mbox{for any}\ 
\varphi \in C^1_c((0,T) \times Q),
\end{equation}
\begin{equation} \label{sc2}
\int_0^T \intQ{ \Big[ \vm \cdot \partial_t \bfphi + 1_{\vr > 0} \frac{\vm \otimes \vm}{\vr}: \Grad \bfphi 
+ p(\vr) \Div \bfphi \Big] } \dt = - \int_0^T \int_Q \Grad \bfphi : \, \D \mathfrak{R} \dt 
\end{equation}
for any $\bfphi \in C^1_c((0,T) \times Q; R^d)$, where $\mathfrak{R}$ is the Reynolds stress, 
\[
\mathfrak{R} \in L^\infty_{\rm weak-(*)}(0,T; \mathcal{M}^+(Q; R^{d \times d}_{\rm sym})),
\]
\begin{equation} \label{sc3}
\expe{ \int_0^T \psi \int_Q \varphi \ \D \ {\rm trace}[\mathfrak{R}] \dt } \leq c \Ov{\mathcal{E}}
\| \psi \|_{L^1(0,T)} \| \varphi \|_{C(Q)}.
\end{equation}
In other words, the limit satisfies  pathwise the compressible Euler system in the generalized sense introduced in  \cite{BreFeiHof19B} with the difference that no energy inequality is postulated. Moreover, since the limit  is a stochastic process, it can be regarded as a statistical dissipative solution in the spirit of \cite{FanFei}.

\subsection{Stochastic Euler system} \label{s:ito}

As the next step, we investigate the question, whether the randomness accumulated  along the statistical vanishing viscosity limit can be directly modeled by a stochastic perturbation in the limiting Euler system.
We suppose that the limit process $(\vr, \vm)$ 
is statistically equivalent, in the sense specified in Definition \ref{iD1}, to 
a weak martingale solution $(\tvr, \tvm)$ of the Euler system driven by the noise 
of It\^o's type: 
\begin{equation} \label{mars}
\D \tvr + \Div \tvm \dt = 0,\ 
\D \tvm + \Div \left( \frac{\tvm \otimes \tvm}{\tvr} \right) \dt + \Grad p(\tvr) \dt = \mathbf{F} \D W.
\end{equation}
Here, $W=(W_{k})_{k\geq 1}$ is a cylindrical Wiener process and the diffusion coefficient $\mathbf{F}=(\mathbf{F}_{k})_{k\geq 1}$ is stochastically integrable, that is, progressively measurable and satisfies 
\begin{equation}\label{eq:101}
\expe{ \int_{0}^{T}\sum_{k\geq 1}\|\mathbf{F}_{k}\|_{W^{-\ell,2}(Q;R^{d})}^{2}\dt }<\infty
\end{equation}
where $W^{-\ell,2}(Q;R^{d})$ is  a possibly negative Sobolev space.  A priori, the coefficient $\mathbf{F}$ may  depend on the solution $(\tvr,\tvm)$ provided the above stochastic integrability condition is satisfied. In this setting, the stochastic integral in \eqref{mars} is a martingale, in particular, it has a zero expectation.

Since we only compare certain statistical properties of the two processes $(\vr,\vm)$ and $(\tvr,\tvm)$, they can be possibly defined on different probability spaces. Nevertheless, without loss of generality we assume for notational simplicity that they are both defined on the probability space $(\Omega,\mathfrak{B},\mathcal{P})$ with expectation $\mathbb{E}$.

It holds $\prst-$a.s. 
\begin{equation} \label{mars1}
\int_0^T \intQ{ \Big[ \tvm \cdot  \bfphi \partial_t \psi + \psi 1_{\tvr > 0} \frac{\tvm \otimes \tvm}{\tvr}: \Grad \bfphi  
+ \psi p(\tvr) \Div \bfphi  \Big] } \dt = -\int_0^T \psi \sum_{k \geq 1} \left( \intQ{ \vc{F}_k \cdot \bfphi } \right) \D W_k 
\end{equation}
for any (deterministic) $\psi \in C^1_c(0,T)$, $\bfphi \in C^1_c(Q; R^d)$. 
Thus passing to expectations, we obtain 
\begin{equation} \label{sc4}
\expe{ \int_0^T \intQ{ \Big[ \tvm \cdot  \bfphi \partial_t \psi + \psi 1_{\tvr > 0} \frac{\tvm \otimes \tvm}{\tvr}: \Grad \bfphi  
+ \psi p(\tvr) \Div \bfphi  \Big] } \dt } = 0.
\end{equation}
Similarly, we consider expectation of the limit equation \eqref{sc2} obtaining
\begin{equation} \label{sc5}
\expe{ \int_0^T \intQ{ \Big[ \vm \cdot  \bfphi\partial_t\psi + \psi 1_{\vr > 0} \frac{\vm \otimes \vm}{\vr}: \Grad \bfphi 
+ \psi p(\vr) \Div \bfphi \Big] } \dt } = -  \expe{ \int_0^T \psi\int_Q \Grad \bfphi : \, \D \mathfrak{R} \dt }. 
\end{equation}
Comparing \eqref{sc4}, \eqref{sc5} and using the fact that the random processes are statistically equivalent
and that $p(\vr) = (\gamma - 1) P(\vr)$,  $p(\tvr) = (\gamma -1 ) P(\tvr)$,
we may infer that 
\begin{equation} \label{sc5a}
\begin{split}
&\expe{\int_0^T \psi \left( \int_Q \Grad \bfphi : \D \mathfrak{R} \right) \dt}  \\
&= \expe{ \int_0^T \psi \intQ{   \left( 1_{\tvr > 0} \frac{\tvm \otimes \tvm}{\tvr} 
- 1_{\vr > 0} \frac{\vm \otimes \vm}{\vr}    \right)
: \Grad \bfphi 
 } \dt } 
\end{split}
\end{equation} 
for any $\psi \in C^1_c(0,T)$, $\bfphi \in C^1_c (Q)$.

Our goal is to show that \eqref{sc5a}, together with the fact that 
$(\vr, \vm)$, $(\tvr, \tvm)$ are statistically equivalent, imply $\mathfrak{R} = 0$ $\prst-$a.s. as soon as $Q$ is an exterior domain 
to a convex body $B$.

\subsubsection{Domains exterior to a convex body}
\label{DPT}

The following result  may be of independent interest.

\begin{Proposition} \label{Psc1}
Let $Q = R^d \setminus B$ where $B$ is a bounded set and let $B_R$ be a (closed) ball in $R^d$ of radius $R$ containing $B$. Suppose 
\[
\mathfrak{R} \in L^\infty_{\rm weak-(*)} (0,T; \mathcal{M}^+(Q, R^{d \times d}_{\rm sym}))
\]
satisfies \eqref{sc5a}, where $(\vr, \vm)$, $(\tvr, \tvm)$ are statistically equivalent in the sense of Definition \ref{iD1}.

Then $\mathfrak{R}|_{R^d \setminus B_R} = 0$ $\prst-$a.s.
\end{Proposition}

\begin{proof}

Without loss of generality, we may suppose that the ball is centered at the origin, $B_R = \{ |x| \leq R \}$. 
\tb{Keep in mind that this amounts to replacing $x$ by $x - x_0$ in the following computations.} Consider a smooth 
convex function 
\[
F(Z) = 0 \ \mbox{for}\ 0 \leq Z \leq R^2,\ 0 < F'(Z) \leq \Ov{F} 
\ \mbox{for}\ R^2 < Z < R^2 + 1,\ F'(Z) = \Ov{F} \ \mbox{if}\ Z \geq R^2 + 1, 
\]
together with
a cut--off function 
\begin{equation} \label{cutof}
\chi \in C^\infty_c[0, \infty), \ \chi(Z) = 1 \ \mbox{for}\ Z \leq 1,\ \chi(Z) = 0 \ \mbox{for} \ Z \geq 2.
\end{equation}

Now, we take  
\[
\bfphi_L (x) = \chi \left( \frac{|x|}{L} \right) \Grad F(|x|^2) ,\ \bfphi \in C^1_c(Q),\ L\geq1.
\]
as a test function in \eqref{sc5a}.
The integral on the left--hand side of \eqref{sc5a} reads
\[
\begin{split}
&\expe{\int_0^T \psi \left( \int_Q \Grad \bfphi_L : \D \mathfrak{R} \right) \dt} = \expe{ \int_0^T \psi \left( \int_{R^d \setminus B_R} \chi \left( \frac{|x|}{L} \right) \nabla_{x}^2 F(|x|^2) 
: \D \mathfrak{R} \right) \dt } \\ &\quad+ \frac{2}{L} \expe{ \int_0^T \psi \left( \int_{L \leq |x| \leq 2L} \chi' \left( \frac{|x|}{L} \right) F'(|x|)|x| \left( \frac{x}{|x|} \otimes \frac{x}{|x|} \right)
: \D \mathfrak{R} \right) \dt },
\end{split}
\]
where, in view of \eqref{sc3}, 
\[
\begin{split}
\frac{2}{L} &\expe{ \int_0^T \psi \left( \int_{L \leq |x| \leq 2L} \chi' \left( \frac{|x|}{L} \right) F'(|x|)|x| \left( \frac{x}{|x|} \otimes \frac{x}{|x|} \right)
: \D \mathfrak{R} \right) \dt } \\
&\leq 4 \expe{ \int_0^T \psi \left( \int_{L \leq |x| \leq 2L} \chi' \left( \frac{|x|}{L} \right) F'(|x|) \left( \frac{x}{|x|} \otimes \frac{x}{|x|} \right)
: \D \mathfrak{R} \right) \dt }
 \to 0 \ \mbox{as}\ L \to \infty.
\end{split}
\]
Computing 
\[
\nabla_{x}^2 F(|x|^2) = 2 \nabla_{x} \left( F'(|x|^2) x \right) = 4 F''(|x|^2) (x \otimes x) + 2 F'(|x|^2) \mathbb{I}
\]
and using convexity of $F$ together with the positive semi--definitness of  $\mathfrak{R}$, we obtain 
\begin{equation} \label{concl1}
\begin{split}
\lim_{L \to \infty} \expe{\int_0^T \psi \left( \int_Q \Grad \bfphi_L : \D \mathfrak{R} \right) \dt}&= \lim_{L \to \infty}\expe{ \int_0^T \psi \left( \int_{R^d \setminus B_R} \chi \left( \frac{|x|}{L} \right) \nabla_{x}^2 F(|x|^2) 
: \D \mathfrak{R} \right) \dt }\\ &\geq 2 \expe{ \int_0^T \psi \int_{R^d \setminus B_R} F'(|x|^2 ) \D \ {\rm trace}[\mathfrak{R}]\dt}
\end{split}
\end{equation}
for any $\psi \in C^1_c(0,T)$, $\psi \geq 0$.

Finally, in accordance with \eqref{SE2}, the integral on the right--hand side of \eqref{sc5a} vanishes. Indeed 
we easily compute
\[
\begin{split}
&\expe{ \int_0^T \psi \intQ{   \left( 1_{\tvr > 0} \frac{\tvm \otimes \tvm}{\tvr} 
- 1_{\vr > 0} \frac{\vm \otimes \vm}{\vr} : \Grad \bfphi_L    \right) } \dt } \\ 
&= \expe{ \int_0^T \psi \intQ{ 4 F''(|x|)  \left( 1_{\tvr > 0} \frac{ |\tvm \cdot x|^2 }{\tvr} 
- 1_{\vr > 0} \frac{|\vm \cdot x|^2 }{\vr}     \right) \chi \left( \frac{|x|}{L} \right) }  \dt }\\ 
& + \expe{ \int_0^T \psi \intQ{ 2 F'(|x|)  \left( 1_{\tvr > 0} \frac{ |\tvm|^2 }{\tvr} 
- 1_{\vr > 0} \frac{|\vm |^2 }{\vr}     \right) \chi \left( \frac{|x|}{L} \right) }  \dt }\\
&+ \frac{2}{L} \expe{ \int_0^T \psi  \int_{L \leq |x| \leq 2L} F'(|x|) \frac{1}{|x|^2}  
\left( 1_{\tvr > 0} \frac{ |\tvm \cdot x|^2 }{\tvr} 
- 1_{\vr > 0} \frac{|\vm \cdot x|^2 }{\vr}     \right)
 \chi' \left( \frac{|x|}{L} \right) \dx \dt } = 0. 
\end{split}
\]
\tb{Indeed using statistical equivalence of the kinetic and angular energies \eqref{SE2} we have}
\[
\begin{split}
&\expe{ \int_0^T \psi \intQ{ 1_{\vr > 0} \frac{|\vm \cdot x|}{\vr} \varphi (x) } \dt } \\ &= 
- \expe{ \int_0^T \psi \intQ{ 1_{\vr > 0} (\mathbb{J}_0 \cdot \vm) \cdot \vm \varphi (x) } \dt } + 
\expe{ \int_0^T \psi \intQ{ 1_{\vr > 0} \frac{|\vm|^2 }{\vr} \varphi(x) } \dt } \\
&= 
- \expe{ \int_0^T \psi \intQ{ 1_{\tvr > 0} (\mathbb{J}_0 \cdot \tvm) \cdot \tvm \varphi (x) } \dt } + 
\expe{ \int_0^T \psi \intQ{ 1_{\tvr > 0} \frac{|\tvm|^2 }{\tvr} \varphi(x) } \dt }\\
&= \expe{ \int_0^T \psi \intQ{ 1_{\tvr > 0} \frac{|\tvm \cdot x|}{\vr} \varphi (x) } \dt }
\end{split}
\]
\tb{for any $\varphi \in C_c(Q)$.}

Going back to \eqref{concl1} we may infer that 
\[
\expe{ \int_0^T \psi \int_{R^d \setminus B_R} F'(|x|^2 ) \D \ {\rm trace}[\mathfrak{R}]\dt} = 0
\]
for any $\psi \in C^1_c(0,T)$, $\psi \geq 0$, which yields the desired conclusion as $F'(|x|^2) > 0$ for $|x| > R$.
\end{proof}

\begin{Corollary} \label{Csc1}

Let $Q = R^d \setminus B$ where $B$ is a compact convex set. 
Suppose 
\[
\mathfrak{R} \in L^\infty_{\rm weak-(*)} (0,T; \mathcal{M}^+(Q, R^{d \times d}_{\rm sym}))
\]
satisfies \eqref{sc5a}, where $(\vr, \vm)$, $(\tvr, \tvm)$ are statistically equivalent in the sense of Definition \ref{iD1}.

Then $\mathfrak{R} = 0$ $\prst-$a.s.

\end{Corollary}

\begin{proof}

As $B$ is convex, any $x \in Q$ possesses an open neighborhood $U(x)$ such that 
\[
U(x) \subset Q \setminus B_R
\]
for some ball $B_R$ containing $B$. By Proposition \ref{Psc1}, 
\[
\mathfrak{R}|_{U(x)} = 0
\]
$\mathcal{P}-$a.s.
\end{proof}

\subsection{Convergence} 

Going back to relation \eqref{R5a}, 
we also obtain 
\begin{equation} \label{sc6}
\Ov{ E \left(\vr, \vm \Big| \vr_\infty , \vm_\infty \right) } = E \left(\vr, \vm \Big| \vr_\infty , \vm_\infty \right)
\ \prst - \mbox{a.s.}
\end{equation}
As shown in \cite{MarEd}, relation \eqref{sc6} implies local strong convergence for the Skorokhod representation,
more specifically, 
\begin{equation} \label{sc8}
 \left\| \tvr_k - \vr \right\|_{L^\gamma((0,T) \times K)}^\gamma + 
\left\| 1_{\tvr_k > 0} \frac{\tvm_k}{\sqrt{\tvr_k}} - 1_{\vr > 0} \frac{\vm}{\sqrt{\vr}} \right\|_{L^2((0,T) \times K; R^d)}^2 \to 0 
\end{equation}
for any compact $ K \subset Q$ $\mathcal{P}$--a.s.

Finally, we translate the convergence result \eqref{sc8} in terms of the original sequence $(\vrn, \vm_n)_{n=1}^\infty$.
Consider $b \in C_c(R^{d + 1})$. It follows from \eqref{sc8} that 
\[
\int_0^T \intQ{ \frac{1}{N_k} \sum_{k=1}^{N_k}  b(\vr_n, \vm_n) \varphi } \dt = 
\expe{ \int_0^T \intQ{ b(\tvr_k, \tvm_k) \varphi } \dt } \to \expe{ \int_0^T \intQ{ b(\vr, \vm) \varphi } \dt }
\]
for any $\varphi \in \DC((0,T) \times Q)$. This can be interpreted as
\[
\frac{1}{N_k} \sum_{k=1}^{N_k}  b(\vr_n, \vm_n) \to \mathbb{E}\left[ b(\vr, \vm) \right] = \mathbb{E}_{\mathcal{V}}\left[ b(r, \vc{w}) \right]
\ \mbox{weakly-(*) in}\ L^\infty((0,T) \times Q) \ \mbox{for}\ b \in C_c(R^{d + 1}).
\] 
Moreover, identifying 
\[
\frac{1}{N_k} \sum_{k=1}^{N_k}  b(\vr_n, \vm_n) = \mathbb{E}\left[ b(\tvr_k, \tvm_k) \right], 
\]
we compute
\[
\left\| 
\mathbb{E}\left[ b(\tvr_k, \tvm_k) \right]- \mathbb{E}\left[b(\vr, \vm) \right] \right\|_{L^1((0,T) \times K)} 
\leq \expe{ \left\| b(\tvr_k, \tvm_k) - b(\vr, \vm) \right\|_{L^1((0,T) \times K)} } \to 0 
\]
for any compact $K \subset Q$.
In particular, modulo a subsequence $(N_k)_{k \geq 1}$, the sequence $(\vrn, \vmn)_{n=1}^\infty$ is S--convergent in the sense of \cite{Fei2020A} as a consequence of Theorem 2.4 in \cite{Fei2020A}.

We have proved the following result. 

\begin{Theorem}[Strong statistical limit] \label{Tsc1}

Suppose that $Q = R^d \setminus B$, where $B$ is \tb{a convex compact} set. Let $(\vr_n, \vm_n)_{n=1}^\infty$ be a sequence of weak solutions to the Navier--Stokes system in the sense of Definition 
\ref{Dns1} with the viscosity coefficients 
\[
\mu_n = \ep_n \mu, \ \lambda_n = \ep_n \lambda , \ \mu > 0, \ \lambda \geq 0, \ \ep_n \!\searrow 0, 
\]
and satisfying the total energy bound \eqref{sl4}. Let $(\vr, \vm)$ be the Skorokhod representation of the 
limit $\mathcal{V}$ identified in Proposition \ref{Psl1}. Suppose that $(\vr, \vm)$  
is statistically equivalent to 
a  solution $(\tvr, \tvm)$ of the 
stochastic Euler system \eqref{mars}, driven by a stochastic forcing $\mathbf{F} \D W$ of It\^ o's type satisfying \eqref{eq:101}.

Then 
\begin{itemize}
\item $(\vr, \vm)$ is a (weak) statistical solution of the deterministic compressible Euler system.
\item The sequence $(\vr_n, \vm_n)_{n=1}^\infty$, modulo a subsequence $(N_k)_{k \geq 1}$, is S--convergent in the sense of 
\cite{Fei2020A}. Specifically, there is a sequence $N_k \to \infty$ such that 
\begin{equation} \label{sc14bis}
\frac{1}{N_k} \sum_{k=1}^{N_k} b(\vr_n, \vm_n) \to \mathbb{E}_{\mathcal{V}} \left[b(r, \vc{w}) \right]
\ \mbox{(strongly) in}\ L^1((0,T) \times K) \ \mbox{as}\ k \to \infty
\end{equation}
for any compact $K \subset Q$ and any $b\in C_c(R^{d + 1})$.

\end{itemize}

\end{Theorem}

\subsection{Drift force of Stratonovich type}

\label{drift}

Revisiting the proof of Theorem~\ref{Tsc1}, we may observe that the only used property of the It\^o integral is its vanishing expectation. In other words, the same result remains valid if we replace the stochastic integral by an arbitrary random variable with zero expectation. A natural question is therefore whether our result applies to other random perturbations with generally non--zero expected value.
Our goal is to extend Theorem \ref{Tsc1} to a larger class of driving forces including, in particular, a Stratonovich type drift term, which is also widely used in the literature. In particular, a physical justification of a noise of this form in the context of fluid dynamics can be found in \cite{MiRo}. In \cite{FL19}, it was even proved that a transport noise of Stratonovich type provides regularization of the incompressible Navier--Stokes system in vorticity form.

For technical reasons, we restrict ourselves to the space dimension $d=2$. 
Unfortunately, we are able to show the result only in the absence of the obstacle, meaning $Q = R^2$.

We suppose that the limit process $(\vr, \vm)$ is statistically equivalent to a solution $(\tvr, \tvm)$ of the problem 
\begin{equation} \label{strat}
\D \tvr + \Div \tvm \dt = 0,\ 
\D \tvm + \Div \left( \frac{\tvm \otimes \tvm}{\tvr} \right) \dt + \Grad p(\tvr) \dt = (\sigma \cdot \Grad) \tvm \circ \D W_1 + \mathbf{F} \ \D W_2,
\end{equation}
where $\sigma \in R^2$ is a constant vector.
In particular, the weak formulation of the momentum equation reads
\begin{equation} \label{strat2}
\begin{split}
&\int_0^T \intQp{ \Big[ \tvm \cdot  \bfphi \partial_t \psi + \psi 1_{\tvr > 0}\frac{\tvm \otimes \tvm}{\tvr}: \Grad \bfphi  
+ \psi p(\tvr) \Div \bfphi  \Big] } \dt\\
 &\qquad= -\int_0^T \psi \sum_{k \geq 1} \left( \intQp{ \vc{F}_k \cdot \bfphi } \right) \D W_{2,k}
+ \int_0^T \psi \left( \intQp{ \tvm\cdot \Grad (\bfphi \otimes \sigma) } \right) \circ \D W_1. 
\end{split}
\end{equation}
for any $\psi \in C^1_c(0,T)$,  \tb{$\bfphi \in C^1_c(R^2; R^2)$}.

Using the It\^o--Stratonovich correction formula we can write the Stratonovich integral in \eqref{strat2} as a martingale plus the correction term
\[
-\frac12\int_0^T \psi \intQp{ \tvm \cdot [ (\sigma \otimes \sigma) : \nabla_x^2] \bfphi } \dt,
\] 
provided
\begin{equation}\label{eq:p1}
\expe{\int_{0}^{T}\left(\psi \int_{R^2}\tvm\cdot\nabla_{x}(\bfphi\otimes\sigma)\ \dx \right)^{2}\dt}<\infty.
\end{equation}
We take this stochastic integrability condition as an additional assumption. 
Obviously, the correction term can be written in the form 
\[
-\frac12\int_0^T \psi \intQp{ \tvm \cdot [ (\sigma \otimes \sigma) : \nabla_x^2] \bfphi } \dt 
= -\frac12\int_0^T \psi \intQp{ (\tvm - \vm_\infty) \cdot [ (\sigma \otimes \sigma) : \nabla_x^2] \bfphi } \dt,
\] 
with $\vm_\infty = \vr_\infty \vu_\infty$.

Similarly to Section \ref{s:ito}, passing to expectations in \eqref{strat2} we obtain 
\begin{equation} \label{sc10}
\begin{split}
\expe{\int_0^T \psi \left( \int_{R^2} \Grad \bfphi : \D \mathfrak{R} \right) \dt} &= \expe{\int_0^T \psi \tb{ \intQp{ (\tvm 
- \vm_\infty) \cdot [ (\sigma \otimes \sigma) : \nabla_x^2] \bfphi } \dt}
}\\
&+\expe{ \int_0^T \psi \intQp{   \left( 1_{\tvr > 0} \frac{\tvm \otimes \tvm}{\tvr} 
- 1_{\vr > 0} \frac{\vm \otimes \vm}{\vr}    \right)
: \Grad \bfphi 
 } \dt } 
\\
\mbox{for any}\ \psi &\in C^1_c(0,T), \ \tb{ \bfphi \in C^1_c (R^2; R^2)}.
\end{split}
\end{equation}

The following steps are inspired by Chae \cite{ChaeD}. 
Similarly to the proof of Proposition \ref{Psc1}, we consider the test function  
\begin{equation} \label{testF}
\bfphi (x) = \chi \left( \frac{|x|}{L} \right) x.  
\end{equation}
where $\chi$ is the cut-off function introduced in \eqref{cutof}. 
Accordingly, 
\[
\Grad \bfphi = \frac{1}{L} \chi' \left( \frac{|x|}{L} \right) \frac{x \otimes x}{|x|} +  
 \chi \left( \frac{|x|}{L} \right) \mathbb{I}, 
\]
\[
\nabla^2_x \bfphi = \frac{1}{L^2} \chi'' \left( \frac{|x|}{L} \right) \frac{x \otimes x}{|x|^2}  
x  - \frac{1}{L}\chi' \left( \frac{|x|}{L} \right)\frac{x \otimes x}{|x|^2} \frac{x}{|x|}     +  \frac{1}{L} \chi' \left( \frac{|x|}{L} \right) \mathbb{I} \frac{x}{|x|} + \frac{1}{L} \chi' \left( \frac{|x|}{L} \right) \frac{1}{|x|} \Grad (x \otimes x). 
\]
Exactly as in the proof of Proposition \ref{Psc1} we deduce
\[
\expe{\int_0^T \psi \left( \int_{R^2} \ \D \ {\rm trace} [\mathfrak{R}] \right) \dt} \leq  
\lim_{L \to \infty} \frac{1}{L} \expe{ \int_0^T \int_{L \leq |x| \leq 2L} \tb{|\tvm - \vm_\infty |} \dx \dt }
\]
Consequently, we obtain the same conclusion as in Proposition \ref{Psc1} if we show 
\begin{equation}\label{eq:ll}
\lim_{L \to \infty} \frac{1}{L} \expe{ \int_0^T \int_{L \leq |x| \leq 2L} \tb{|\tvm - \vm_\infty |} \dx \dt } = 0. 
\end{equation}

We check easily by direct manipulation that 
\[
\begin{split}
1_{ \left\{ \frac{1}{2} \vr_\infty \leq \tvr \leq 2 \vr_\infty \right\} }
\tb{ |\tvm - \vm_\infty|^2} &\aleq E\left( \tvr, \tvm \Big| \vr_\infty, \vu_\infty \right) ,\\
 1_{ \left\{ \tvr < \frac{1}{2} \vr_\infty \ \rm{or}\ \tvr > 2 \vr_\infty \right\} }
\tb{|\tvm - \vm_\infty|^{\frac{2 \gamma}{\gamma + 1}} } &\aleq E\left( \tvr, \tvm \Big| \vr_\infty, \vu_\infty \right) .
\end{split}
\]
Denoting 
\[
\tvm_1 = 1_{ \left\{ \frac{1}{2} \vr_\infty \leq \tvr \leq 2 \vr_\infty \right\} } \tb{ (\tvm - \vm_\infty) } , 
\ \tvm_2 = 1_{ \left\{ \tvr < \frac{1}{2} \vr_\infty \ \rm{or}\ \tvr > 2 \vr_\infty \right\} } \tb{ (\tvm - \vm_\infty) }, 
\]
we get, by H\" older's inequality,  
\[
\frac{1}{L} \expe{ \int_0^T \int_{L \leq |x| \leq 2L} |\tvm_1 | \dx \dt } 
\aleq L^{\frac{d - 2}{2}} \expe{  \int_0^T \left( \int_{L \leq |x| \leq 2L} E\left( \tvr, \tvm \Big| \vr_\infty, \vu_\infty \right) 
\dx \right)^{\frac{1}{2}} \dt },
\]
and 
\[
\frac{1}{L} \expe{ \int_0^T \int_{L \leq |x| \leq 2L} |\tvm_2 | \dx \dt } 
\aleq L^{d \frac{\gamma - 1}{2 \gamma} - 1} \expe{  \int_0^T \left( \int_{L \leq |x| \leq 2L} E\left( \tvr, \tvm \Big| \vr_\infty, \vu_\infty \right) 
\dx \right)^{\frac{\gamma + 1}{2\gamma}} \dt }.
\]
However, as $(\tvr, \tvm)$ is statistically equivalent to $(\vr, \vm)$ we have 
\[
\expe{ \int_0^T \int_{L \leq |x| \leq 2L} E\left( \tvr, \tvm \Big| \vr_\infty, \vu_\infty \right) 
\dx \dt } 
= \expe{ \int_0^T \int_{L \leq |x| \leq 2L} E\left( \vr, \vm \Big| \vr_\infty, \vu_\infty \right) 
\dx \dt }
\]
In particular, if $d = 2$, both integrals vanish in the asymptotic limit $L \to \infty$.

We have proved the following result. 

\begin{Theorem} \label{Tsc2}

\tb{Let $Q = R^2$}. Suppose that $(\vr_n, \vm_n)_{n=1}^\infty$ is a sequence of weak solutions to the Navier--Stokes system in the sense of Definition 
\ref{Dns1} with the viscosity coefficients 
\[
\mu_n = \ep_n \mu, \ \lambda_n = \ep_n \lambda , \ \mu > 0, \ \lambda \geq 0, \ \ep_n \!\searrow 0, 
\]
and satisfying the total energy bound \eqref{sl4}. Let $(\vr, \vm)$ be the Skorokhod representation of the 
limit $\mathcal{V}$ identified in Proposition \ref{Psl1}. Suppose that $(\vr, \vm)$ is statistically equivalent to a weak solution 
$(\tvr, \tvm)$
of the 
stochastic Euler system \eqref{strat} driven by a stochastic forcing  
\[
(\sigma \cdot \Grad) \tvm \, \circ \D W_1 + \mathbf{F} \,\D W_2
\]
and satisfying \eqref{eq:p1} for any $\psi\in C^1_c(0,T),$ \tb{ $\bfphi \in C^1_c (R^2; R^2)$}.

Then 
\begin{itemize}
\item $(\vr, \vm)$ is a (weak) statistical solution of the deterministic compressible Euler system.
\item The sequence $(\vr_n, \vm_n)_{n=1}^\infty$, modulo a subsequence $(N_k)_{k \geq 1}$, is S--convergent in the sense of 
\cite{Fei2020A}. Specifically, there is a sequence $N_k \to \infty$ such that 
\[
\frac{1}{N_k} \sum_{k=1}^{N_k} b(\vr_n, \vm_n) \to \mathbb{E}_{\mathcal{V}}\left[ b(r, \vc{w}) \right]
\ \mbox{(strongly) in}\ L^1((0,T) \times K) \ \mbox{as}\ k \to \infty
\]
for any compact \tb{$K \subset R^2$} and any  $b\in C_c(R^{3})$.

\end{itemize}

\end{Theorem}

Similarly to the It\^o case, we may extend the validity of Theorem~\ref{Tsc2} to random perturbations of the form
$
M +N
$, where $M$ has zero expectation and $N$ satisfies a suitable assumption so that the corresponding term  vanishes in the required sense as it was the case for the It\^o--Stratonovich correction through \eqref{eq:ll}.

\begin{Remark} \label{rr2}

\tb{
A short inspection of formula \eqref{sc10} and the specific form of the test function \eqref{testF} reveals that 
the hypothesis concerning statistical equivalence of the angular energies can be relaxed in \eqref{SE2}.
}

\end{Remark}

\begin{Remark} \label{rr1}

\tb{
It is interesting to note that the same technique can be used to eliminate system with ``effective'' viscosity 
(cf. Davidson \cite[Chapter 4]{DAVI}), }
\[
\partial_t \tvr + \Div \tvm = 0,\ 
\partial_t \tvm + \Div \left( \frac{\tvm \otimes \tvm}{\tvr} \right)  + \Grad p(\tvr) \dt = \mu_{\rm eff} \Del \tvu,\ 
\]
\tb{
as the turbulent limit provided the velocity field $\tvu$ enjoys certain integrability properties.}

\end{Remark}

\subsubsection{Three dimensional setting}

We observe that the proof of Theorem~\ref{Tsc2} goes through also in three spatial dimensions provided certain stronger assumptions are postulated. \tb{We assume that the far field density vanishes $\vr_{\infty}=0$, and that the adiabatic constant  
belongs to the physically relevant range $\gamma\in (1,3]$.} In that case, the term $\widetilde{\vm}_{1}$ does not appear and the bound for $\widetilde{\vm}_{2}$ indeed vanishes as $L\to\infty$, which yields the claim of Theorem~\ref{Tsc2}.

\tb{ We can also allow for more complicated spatial domains provided suitable boundary conditions for both the approximate and the 
limit system are considered.
As the spatial domain we may take for instance a domain exterior to a cone,}
\begin{equation*}
\tb{Q = R^3 \setminus Q},\ 
Q = \left\{ (x_1, x_2,x_{3}) \ \Big| x_{3}\geq 0,\  x^{2}_{3}\leq \lambda^{2}(x_{1}^{2}+x_{2}^{2}) \right\},\ \lambda>0.
\end{equation*}
\tb{We suppose that the solutions of the Navier--Stokes system satisfy the complete slip condition,} 
\[
\vu \cdot \vc{n} |_{\partial Q} = 0, \ (\mathbb{S} \cdot \vc{n}) \times \vc{n}|_{\partial Q} = 0.
\]
\tb{Accordingly, the limit Euler system is endowed with the standard impermeability condition} 
\[
\vm \cdot \vc{n}|_{\partial Q} = 0.
\] 
Under these circumstances, the function $\bfphi$ introduced in \eqref{testF} is still an admissible test function for both 
the Navier--Stokes and the Euler system, and the arguments of Section \ref{drift} can be used as soon as $\vr_\infty = 0$, 
$\vu_\infty = 0$. Clearly, such a choice of boundary conditions eliminates the turbulent boundary layer. The same can be done also in two spatial dimensions where the assumption $\vr_{\infty}=0$ is not necessary.

\subsection{Deterministic limit}

As our final goal, we identify the asymptotic limit in the situation of Theorem~\ref{Tsc1} and under the additional assumption that  the barycenter
\[
(\overline\vr, \overline\vm) = \mathbb{E}_{\mathcal{V}}\left[ (r, \vc{w}) \right]\]
is a weak solution of the Euler system.

Consider the averages 
\[
\frac{1}{N_k} \sum_{n=1}^{N_k} \vr_n = \Ov{\vr}_k,\ \frac{1}{N_k} \sum_{n=1}^{N_k} \vm_n = \Ov{\vm}_k. 
\]
Due to convexity of the relative energy, we have 
\[
\intQ{ E \left(\Ov{\vr}_k, \Ov{\vm}_k \Big| \vr_\infty, \vu_\infty \right) } \leq 
\frac{1}{N_k} \sum_{n=1}^{N_k} \intQ{ E \left(\vr_n, \vm_n \Big| \vr_\infty, \vu_\infty \right) }, 
\]
and the uniform bound \eqref{sl4} yields 
\begin{equation} \label{sc14}
\sup_{0 \leq t \leq \tau} \intQ{E \left(\Ov{\vr}_k, \Ov{\vm}_k \Big| \vr_\infty, \vu_\infty \right) } \leq \Ov{\mathcal{E}}. 
\end{equation}
Thus it follows from \eqref{sc14bis}  that 
\[
\Ov{\vr}_k \to \overline\vr \ \mbox{in}\ L^p_{\rm loc}([0,T] \times Q),\ 1 \leq p < \gamma, \  
\Ov{\vm}_k \to \overline\vm \ \mbox{in}\ L^q_{\rm loc}([0,T] \times Q; R^d),\ 1 \leq q < \frac{2 \gamma}{\gamma + 1}.
\]

As $(\vr_n, \vm_n)$ satisfy the momentum equation, we get 
\[
\begin{split}
\int_0^T &\intQ{ \Ov{\vm}_k \cdot \partial_t \bfphi + 1_{\Ov{\vr}_k > 0} \frac{\Ov{\vm}_k \otimes \Ov{\vm}_k }{\Ov{\vr}_k} 
: \Grad \bfphi + p(\Ov{\vr}_k) \Div \bfphi } \dt \\ 
&= \frac{1}{N_k} \sum_{n=1}^{N_k} \ep_n \int_0^T \intQ{ \mathbb{S}(\Grad \vu_n) : \Grad \bfphi } \dt \\ 
&- \int_0^T \intQ{ \left( \frac{1}{N_k} \sum_{n=1}^{N_k}\left[ 1_{\vr_n > 0} \frac{ \vm_n \otimes \vm_n }{\vr_n} + p(\vr_n) \mathbb{I} \right]  - \left[ 1_{\Ov{\vr}_k > 0} \frac{\Ov{\vm}_k \otimes \Ov{\vm}_k }{\Ov{\vr}_k} 
 + p(\Ov{\vr}_k) \mathbb{I} \right] \right) : \Grad \bfphi } \dt, 
\end{split}
\]
or, as $(\overline\vr, \overline\vm)$ is a weak solution of the Euler system,
\[
\begin{split} 
\int_0^T &\intQ{ \left( \frac{1}{N_k} \sum_{n=1}^{N_k}\left[ 1_{\vr_n > 0} \frac{ \vm_n \otimes \vm_n }{\vr_n} + p(\vr_n) \mathbb{I} \right]  - \left[ 1_{\Ov{\vr}_k > 0} \frac{\Ov{\vm}_k \otimes \Ov{\vm}_k }{\Ov{\vr}_k} 
 + p(\Ov{\vr}_k) \mathbb{I} \right] \right) : \Grad \bfphi } \dt\\
&= \int_0^T \intQ{ \left(  \left[ 1_{\Ov{\vr}_k > 0} \frac{\Ov{\vm}_k \otimes \Ov{\vm}_k }{\Ov{\vr}_k} 
 + p(\Ov{\vr}_k) \mathbb{I} \right] - \left[ 1_{{\overline\vr} > 0} \frac{{\overline\vm} \otimes {\overline\vm} }{{\overline\vr}} 
 + p(\overline\vr) \mathbb{I} \right] \right) : \Grad \bfphi } \dt\\
&+ \frac{1}{N_k} \sum_{n=1}^{N_k} \ep_n \int_0^T \intQ{ \mathbb{S}(\Grad \vu_n) : \Grad \bfphi } \dt 
+ \int_0^T \intQ{ (\overline\vm - \Ov{\vm}_k) \cdot \partial_t \bfphi } \dt \to 0 \ \mbox{as}\ k \to \infty
\end{split}
\] 
for any $\bfphi \in C^1_c((0,T) \times Q)$. 

Next, observe that the tensor 
\[
\mathfrak{R}^1_k = \left( \frac{1}{N_k} \sum_{n=1}^{N_k}\left[ 1_{\vr_n > 0} \frac{ \vm_n \otimes \vm_n }{\vr_n} + p(\vr_n) \mathbb{I} \right]  - \left[ 1_{\Ov{\vr}_k > 0} \frac{\Ov{\vm}_k \otimes \Ov{\vm}_k }{\Ov{\vr}_k} 
 + p(\Ov{\vr}_k) \mathbb{I} \right] \right)
\]
is positively semi--definite as 
\[
\begin{split}
&\left( \frac{1}{N_k} \sum_{n=1}^{N_k}\left[ 1_{\vr_n > 0} \frac{ \vm_n \otimes \vm_n }{\vr_n} + p(\vr_n) \mathbb{I} \right]  - \left[ 1_{\Ov{\vr}_k > 0} \frac{\Ov{\vm}_k \otimes \Ov{\vm}_k }{\Ov{\vr}_k} 
 + p(\Ov{\vr}_k) \mathbb{I} \right] \right) : (\xi \otimes \xi)\\
&= \left( \frac{1}{N_k} \sum_{n=1}^{N_k}\left[ 1_{\vr_n > 0} \frac{ |\vm_n \cdot \xi|^2 }{\vr_n} + p(\vr_n) |\xi|^2 \right]  - \left[ 1_{\Ov{\vr}_k > 0} \frac{|\Ov{\vm}_k \cdot \xi|^2 }{\Ov{\vr}_k} 
 + p(\Ov{\vr}_k) |\xi|^2 \right] \right) \geq 0
\end{split}
\]
in view of convexity of the function 
\[
(\vr, \vm) \mapsto \frac{|\vm \cdot \xi|^2}{\vr} + p(\vr)|\xi|^2.
\]
Note that boundedness of the energy implies 
\[
\frac{|\vm_n|^2}{\vr_n} = 1_{\vr_n > 0} \frac{|\vm_n|^2}{\vr_n},\ \frac{|\Ov{\vm}_k|^2}{\Ov{\vr}_k} = 1_{\Ov{\vr}_k > 0} \frac{|
\Ov{\vm}_k|^2}{\Ov{\vr}_k}
 \ \mbox{a.a.}
\]

We show that 
\[
{\rm ess} \sup_{0 \leq \tau \leq T} \intQ{ {\rm trace} [\mathfrak{R}^1_k ] } \leq c(\Ov{\mathcal{E}}).
\]
As we have observed above, this is equivalent to a similar statement for the energy, namely
\[
{\rm ess} \sup_{0 \leq \tau \leq T} \intQ{ \left( \frac{1}{N_k} \sum_{k=1}^{N_k}  E(\vr_n, \vm_n)  - 
E(\Ov{\vr}_k, \Ov{\vm}_k ) \right) }\leq c(\overline{\mathcal{E}}).
\]
However, a simple computation yields 
\[
\frac{1}{N_k} \sum_{k=1}^{N_k}  E(\vr_n, \vm_n)  - 
E(\Ov{\vr}_k, \Ov{\vm}_k ) = \frac{1}{N_k} \sum_{k=1}^{N_k}  E \left (\vr_n, \vm_n \Big| \vr_\infty, \vm_\infty \right)
- E\left(\Ov{\vr}_k, \Ov{\vm}_k \Big| \vr_\infty, \vm_\infty \right),
\]
and the desired conclusion follows from the energy bound \eqref{sl4}.

Repeating the arguments of Section \ref{DPT} we conclude 
\begin{equation} \label{final}
\begin{aligned}
&\lim_{k \to \infty} \int_0^T \intQ{ \frac{1}{N_k} \sum_{k=1}^{N_k} E(\vr_n, \vm_n) \varphi } \dt  = 
\int_0^T \intQ{  E(\overline\vr, \overline\vm) \varphi } \dt \\
&\qquad\qquad \mbox{for any}\ \varphi \in C_c((0,T) \times R^d),\ 
\varphi \geq 0.
\end{aligned}
\end{equation} 
As shown in \cite[Section 5]{Fei2020A}, relation \eqref{final} implies 
\[
\lim_{k \to \infty}  \frac{1}{N_k} \sum_{k=1}^{N_k} b(\vr_n, \vm_n)  \to b(\overline\vr, \overline\vm) 
\]
for any $b \in C_c(R^{d+1})$. This yields the final conclusion 
\[
\frac{1}{N_k} \sum_{n=1}^{N_k} \Big( \| \vr_n - \overline\vr \|_{L^1((0,T) \times K)} + 
\| \vm_n - \overline\vm \|_{L^1((0,T) \times K; R^d)} \Big) \to 0 \ \mbox{as}\ k \to \infty. 
\]
Consequently,  $\mathcal{V} = \delta_{(\overline\vr,\overline \vm)}$ and $(\vrn, \vmn)_{n=1}^\infty$ statistically converges to $(\overline\vr,\overline \vm)$.

\begin{Theorem} \label{Tsc3}
In addition to the hypotheses of Theorem \ref{Tsc1}, suppose that 
\[
(\overline\vr,\overline \vm) = \mathbb{E}_{\mathcal{V}}\left[ (r, \vc{w}) \right] 
\]
is a weak solution to the Euler system. 

Then 
\[
\mathcal{V} = \delta_{(\overline\vr,\overline \vm)}
\]
and  
\begin{equation} \label{stat}
\frac{1}{N_k} \sum_{n=1}^{N_k} \Big( \| \vr_n - \overline\vr \|_{L^1((0,T) \times K)} + 
\| \vm_n - \overline\vm \|_{L^1((0,T) \times K; R^d)} \Big) \to 0 \ \mbox{as}\ k \to \infty. 
\end{equation}

\end{Theorem}

As shown by Connor \cite{Conn}, relation \eqref{stat} yields statistical convergence, modulo the subsequence $(N_k)_{k \geq 1}$, of the sequence $(\vr_n, \vm_n)$:
For any $\ep > 0$, 
\[
\frac{\# \left\{ n \leq N_k \ \Big| \Big( \| \vr_n - \overline\vr \|_{L^1((0,T) \times K)} + 
\| \vm_n - \overline\vm \|_{L^1((0,T) \times K; R^d)} \Big) > \ep \right\} } {N_k} \to 0 \ \mbox{as}\ k \to \infty.
\]

\section{Concluding remarks}
\label{D}

We have studied the vanishing viscosity limit for the compressible Navier--Stokes fluid flow around a convex obstacle. 
We have shown that the statistical limit cannot be a solution of the associated Euler system driven by stochastic forcing of It\^ o's 
type. As a consequence,  there are two basic scenarios to describe the statistical limit:
\begin{itemize}
\item 
{\bf Oscillatory limit.} The limit is in the weak sense and can be described in terms of a Young measure. In that case,   
the weak limit \emph{is not} a weak solution of the Euler system. Statistically, however, the asymptotic limit is a singleton (Dirac measure). Note that this scenario is \emph{compatible} with the hypothesis that the limit is independent of the choice of $\ep_n \searrow 0$.

\item 
{\bf Statistical limit.} The limit is a statistical solution of the Euler system. If this is the case, the sequence 
of the Navier--Stokes system S--converges in the sense of \cite{Fei2020A}.
This is apparently in agreement with the numerical experiments performed by Fjordholm et al. \cite{FjKaMiTa}.  At the theoretical level, this alternative is also in agreement with the Kolmogorov hypothesis concerning turbulent flow 
that predicts compactness in the strong Lebesgue sense, see Chen and Glimm \cite{CheGli}.  We point out that 
this scenario \emph{is not compatible} with the hypothesis that the limit is independent of 
$\ep_n \searrow 0$ unless it is a monoatomic measure in which case the convergence must be strong.

\end{itemize}

The phenomena of oscillatory and statistical limit are in a way complementary. The former 
asserts weak (oscillatory) convergence to a single limit, the latter means strong convergence for any suitable subsequence but 
non--existence of a single limit. In reality, they can be, of course, mixed up. In both cases, the generating sequence of solutions 
of the Navier--Stokes system is S--convergent in the sense of \cite{Fei2020A}.

The results depend essentially on the properties of the \emph{compressible} Euler system, and also on the geometry of the physical space -- an unbounded  domain exterior to a convex set in $R^d$. They might be extended to the case of the full Euler system, however, 
the zero viscosity limit here is more delicate as the existence theory for the Navier--Stokes--Fourier system is available only for a particular class of constitutive equations.  
 
\def\cprime{$'$} \def\ocirc#1{\ifmmode\setbox0=\hbox{$#1$}\dimen0=\ht0
  \advance\dimen0 by1pt\rlap{\hbox to\wd0{\hss\raise\dimen0
  \hbox{\hskip.2em$\scriptscriptstyle\circ$}\hss}}#1\else {\accent"17 #1}\fi}


\end{document}